\documentclass[a4paper,reqno, 10pt]{amsart}
\usepackage{amsmath,amssymb,amsthm,enumerate}

\usepackage{paralist}
\usepackage{graphics}
\usepackage{epsfig}
\usepackage{amsfonts}
\usepackage{amssymb}
\usepackage{color}

\usepackage{mathtools}

\usepackage[numbers,sort&compress]{natbib}

\usepackage{mathrsfs}
\usepackage{lineno}

\usepackage[svgnames]{xcolor}

\usepackage{pifont}
\usepackage{mathtools}

\usepackage[misc]{ifsym}

\usepackage[margin=0.9in]{geometry}
\def\ifl{\iffalse }

\def\bc{\begin{center}} \def\ec{\end{center}}
\def\ba{\begin{array}} \def\ea{\end{array}}
\def\bea{\begin{eqnarray}} \def\eea{\end{eqnarray}}
\def\beaa{\begin{eqnarray*}} \def\eeaa{\end{eqnarray*}}

\newtheorem{thm}{Theorem}[section]

\newtheorem{lem}{Lemma}[section]
\theoremstyle{remark}
\newtheorem{rem}{Remark}[section]
\newtheorem*{rem*}{Remark}

\numberwithin{equation}{section}

\newcommand{\R}{\mathbb{R}}

\newcommand{\pa}{\partial}
\newcommand{\na}{\nabla}
\newcommand{\al}{\alpha}
\newcommand{\be}{\beta}

\newcommand{\Ga}{\Gamma}

\renewcommand{\bar}[1]{\overline{#1}}

\newcommand{\Lg}{\langle}
\newcommand{\Rg}{\rangle}

\newcommand{\OK}{\operatorname{OK}}

\newcommand{\Flor}[1]{\lfloor{#1}\rfloor}

\title[Quasilinear wave]{uniform estimates for  2D quasilinear wave}

\author[D. Li]{ Dong Li}
\address{D. Li, Department of Mathematics, The University of Hong Kong, Hong Kong, China}%
\email{mathdl@hku.hk}

\begin{document}

\begin{abstract}
We consider two-dimensional quasilinear wave equations with standard null-type quadratic nonlinearities. In 2001 Alinhac proved that such systems possess global in time solutions for
 compactly supported initial data with sufficiently small Sobolev norm. The highest norm of
the constructed solution grows polynomially in time.  In this work we develop a new
strategy and  prove uniform boundedness of the highest order norm of the solution for all time.
\end{abstract}

\maketitle

\section{introduction}
Let $\square = \partial_{tt} - \Delta =\partial_{tt} -\partial_{x_1x_1} -\partial_{x_2x_2}$ be the usual d'Alembertian operator in ($2+1$) space-time. 
 We consider the following quasilinear wave equation:
\begin{align}\label{eq:we2d}
 \begin{cases}
  \square u=g^{kij} \partial_k u \partial_{ij} u,\, \quad t>2, \quad x\in\R^2;\\
  (u, \partial_t u)\Bigr|_{t=2} = (f_1, f_2),
 \end{cases}
\end{align}
where the functions $f_1$, $f_2:\; \mathbb R^2\to \mathbb R$ are initial data.
On the RHS of \eqref{eq:we2d} we employ the Einstein summation convention with $\partial_0=\partial_t$, $\partial_l=
\partial_{x_l}$, $l=1,2$.  For simplicity we assume $g^{kij}$ are constant coefficients,
$g^{kij}=g^{kji}$ for any $i$, $j$,  and satisfy the standard null condition:
\begin{align} \label{null1}
g^{kij} \omega_k \omega_i \omega_j=0, 
\qquad\text{for any null $\omega$, i.e., 
$\omega=(-1,\cos\theta, \sin \theta)$, $\theta \in [0,2\pi]$}.
\end{align}
In the seminal work \cite{Alinh01_1},  Alinhac   showed that for compactly supported  initial data  which have sufficiently small Sobolev norm, 
 the system \eqref{eq:we2d} with the null condition \eqref{null1} admits global in time solutions.
 The main ingredients of Alinhac's proof are two: 
 \begin{itemize}
\item[1)] construction of an approximation solution;
\item[2)] time-dependent weighted energy estimates known as the ghost weight method.  
\end{itemize}
The name ghost weight derives from a judiciously chosen bounded space-time weight which seems negligible by itself but after differentiation produces a remarkable
 stabilization term helping to balance the critical decay of the solution near the light-cone. Besides
 the aforementioned ghost weight,  the weighted energy estimates  typically involve a number
of vector fields which are the infinitesimal generators of certain symmetry groups, for example:
\begin{itemize}
\item Spatial rotation: $\partial_{\theta}= x^{\perp}\cdot \nabla = x_1 \partial_2 -x_2 \partial_1$,
$x^{\perp}=(-x_2,x_1)$.
\item Lorentz boost: $L_i = x_i \partial_0 + t \partial_i$, $i=1,2$.
\item Scaling:  $L_0= t \partial_t + r \partial_r$.
\end{itemize}
In particular,  the Lorentz boost vector fields were employed together with the scaling operator in order to extract sufficient time-decay of the solution. While the Lorentz boost vector fields can lead to strong time-decay estimates, they are not suitable for general wave systems which are not Lorentz invariant. To name a few we mention  non-relativistic wave systems with multiple wave speeds 
(cf. \cite{SK96,ST01}), nonlinear wave equations
on non-flat space-time (cf. \cite{Yang16}) and exterior domains (cf. \cite{MetSog07}).
From this perspective it is of fundamental importance to remove the Lorentz boost operator
and develop a new strategy for the general non-Lorentz-invariant systems. 
In \cite{Hosh06}, Hoshiga considered a quasilinear system with multiple speeds of propagation,
and proved global wellposedness under some suitable strong null conditions. In \cite{Zha16} (see also  \cite{PZ16}), Zha considered \eqref{eq:we2d}--\eqref{null1} with the  additional symmetry condition:
$
g^{kij}= g^{ikj}=g^{jik}, \;\forall\, i,j,k.
$
For this case Zha developed the first proof of  global wellposedness without using the Lorentz boost vector fields. 
Note that the additional symmetry condition introduced by Zha appears to be a bit restrictive. For example, it does not include
the standard nonlinearity $\partial ( |\partial_t u|^2 - |\nabla u|^2)$. In recent work \cite{CLX},
a novel strong null form which includes several prototypical
strong null forms such as $\partial ( |\partial_t u|^2 - |\nabla u|^2)$ in the literature 
and also some null forms in \cite{CLM18} as special cases. Moreover  for this class of new null forms,  a new normal-form type Lorentz-boost-free
strategy was developed in \cite{CLX} to prove global wellposedness and uniform boundedness of the highest norm of the solution. We refer to the papers \cite{Lei16, CLM18, Cai21, Dong21, HY20, Zh19, Ka17, Chrd86, Hor86, John90, Sid97, Sid00,WangYu14} for other related developments and different strategies.

We now mention a few other important works on somewhat related systems.
 In \cite{Lei16}, by using Alinhac's method, Lei established small data global wellposedness for 2D incompressible elastodynamics.  A similar result was obtained independently by X. Wang in \cite{W17} using a normal form method.
 In \cite{CLM18}, 
 Cai, Lei and Masmoudi considered the quasilinear wave equations of the form
$
\square u = A_l \partial_l ( N_{ij} \partial_i u \partial_j u)$, 
where $A_l$, $N_{ij}$ are constants, and 
$
N_{ij} \omega_i \omega_j=0
$
for any null vector $\omega$.
A special case is the equation $
\square u = \partial_t ( |\partial_t u|^2 - |\nabla u|^2)$.
In \cite{CLM18} by using a nonlocal transformation (see Remark 1.3 therein) 
it was shown that the above system  has a uniform bound of the highest-order
energy for all time. More recently by using Alinhac's ghost weight
and the null structure in the Lagrangian formulation, Cai \cite{Cai21} showed uniform boundedness of
the highest-order energy for 2D incompressible elastodynamics. In \cite{Dong21}, by using the hyperbolic
foliation method which goes back to H\"ormander and Klainerman, Dong, LeFloch and Lei
showed that the top-order energy of the system \eqref{eq:we2d} with the null condition
\eqref{null1} is uniformly bounded for all time. The main advantage
of the hyperbolic change of variable is that one can gain better control of the conformal
energy thanks to the extra integrability in the hyperbolic time $s=\sqrt{t^2-r^2}$. One should note,
however, that if one works  with the advanced coordinate $s=t-r$, then there is certain degeneracy 
in the $\partial_s$ direction which renders (even any generalized)
conformal energy out of control.  In this connection an interesting further issue is to explore the monotonicity of the conformal energy (and possible generalizations) with respect to different space-time foliations.  Another subtle technical issue in the hyperbolic foliation method is the  extensive use of Lorentz boost
vector fields which appears not suitable for general non-Lorentz-invariant systems.

In Alinhac's work \cite{Alinh01_1}, the highest energy of the
constructed solution has an upper bound which grows polynomially in time. An ensuing
open question is whether this growth is genuinely true phenomena known as  the ``blowup-at-infinity"
conjecture (\cite{Al02, Al03, Al10}).  The purpose of this work is to develop further the program initiated in \cite{Zha16, PZ16, CLX, CLXZ}, and obtain the uniform boundedness of highest norm under the generic null condition \eqref{null1}. We introduce a new approach and settle the blowup-at-infinity conjecture  without employing the Lorentz boost vector fields. 
The main result is the following.
\begin{thm}\label{thm:main}
Consider \eqref{eq:we2d} with $g^{kij}$ satisfying the standard null condition 
\eqref{null1}. Let $m\ge 5$ and assume $f_1 \in H^{m+1}(\mathbb R^2) $, $f_2\in H^{m} (\mathbb R^2)$ are compactly supported in the disk $\{|x| \le 1\}$.  There exists $\varepsilon_{0}>0$ depending on $g^{kij}$ and $m$ such that if
 $\|f_1\|_{H^{m+1}}+\|f_2\|_{H^{m}} <\varepsilon_0 $, then the system \eqref{eq:we2d} has a unique global solution. 
Furthermore, the highest norm of the solution remains uniformly bounded, namely
\begin{align}
\sup_{t\ge 2} \sum_{|\alpha|\le m}  \| (\partial \Gamma^{\alpha} u)(t,\cdot) \|_{L_x^2(\mathbb R^2)}
< \infty.
\end{align}
Here $\Gamma=\{ \partial_t, \partial_{x_1}, \partial_{x_2}, \partial_{\theta}, t\partial_t +r\partial_r\}$ does not include the Lorentz boost (see  \eqref{def_Gamma1} for notation). 
\end{thm}
\begin{rem}
The regularity constraint $m\ge 5$ can  be lowered further by optimizing some technical arguments. However we shall not dwell on this issue in this work.
\end{rem}

We now explain the key steps of the proof of Theorem \ref{thm:main} (see section 2
for the relevant notation). Fix any multi-index $\alpha$ with $|\alpha|\le m$
and consider $\Gamma^{\alpha} u$.   By Lemma \ref{Lem2.3}, we have
\begin{align} \label{1.6AA}
\square \Gamma^{\alpha} u = \sum_{\alpha_1+\alpha_2 \le \alpha}  g^{kij}_{\alpha;\alpha_1,\alpha_2} 
\partial_k \Gamma^{\alpha_1}u \partial_{ij} \Gamma^{\alpha_2} u,
\end{align}
where $g^{kij}_{\alpha;\alpha,0} =g^{kij}_{\alpha;0,\alpha}=g^{kij}$,  and $g^{kij}_{\alpha;
\alpha_1,\alpha_2}$ still satisfies the null condition \eqref{null1} for all other values of
$(\alpha_1,\alpha_2)$.

Step 1. Weighted energy estimates: LHS of \eqref{1.6AA}. We choose $p (r,t)=q(r-t)$ with $q^{\prime}(s)$ nearly
scales as $\langle s \rangle^{-1}$ to derive
\begin{align}
\int \square \Gamma^{\alpha}u \partial_t \Gamma^{\alpha} u e^p dx = 
\frac 12 \frac d {dt}
( \| e^{\frac p2} \partial \Gamma^{\alpha} u \|_2^2)
+ \frac 12 \int e^p q^{\prime} |T \Gamma^{\alpha} u|^2 dx.
\end{align}
Summing over $|\alpha|\le m$, we have (below $v=\Gamma^{\alpha} u$)
\begin{align}
&\sum_{|\alpha| \le m} \| e^{\frac p2} \partial v \|_2^2 \sim E_{m} 
=\sum_{|\alpha|\le m} \| \partial \Gamma^{\alpha} u(t,\cdot) \|_{2}^2; \\
&\sum_{|\alpha|\le m}  \int e^p q^{\prime} |Tv|^2 dx
= \sum_{|\alpha|\le m} 
\int e^p q^{\prime} |T \Gamma^{\alpha} u |^2 dx. 
\end{align}

Step 2. Refined decay estimates. One way to remedy the lack of Lorentz boost vector fields
is to employ $L^{\infty}$ and $L^2$ estimates involving the weight-factor $\langle
r-t\rangle$. At the expense of certain smallness of $E_{\Flor{\frac m2}+3}$ and
using in an essential way the nonlinear null form
(see Lemma \ref{lem2.3a}), we obtain
\begin{align}
& \| \langle r -t\rangle (\partial^2 \Gamma^{\le l_0}
u)(t, \cdot)\|_2 \lesssim \| (\partial \Gamma^{\le l_0+1} u )(t,\cdot) \|_2,\qquad
\forall\, l_0\le m-1; \\
& |\langle r -t \rangle (\partial^2 \Gamma^{\le l_0}
u)(t,x)| \lesssim | (\partial \Gamma^{l_0+1} )(t,x)|, \qquad\forall\, r\ge t/10, \,
l_0\le m-1; \\
& \| (\partial  \Delta \Gamma^{\le m-3} u )(t, \cdot )\|_{L_x^2(|x|\le \frac 23 t)} \lesssim
t^{-2} \| (\partial \Gamma^{\le m-1} u)(t,\cdot) \|_2. 
\end{align} 
These in turn lead to a handful of new strong decay estimates (see Lemma 
\ref{lem2.6}):
\begin{align}
&t^{\frac 12}
\| \partial \Gamma^{\le m-3} u\|_{\infty}
+t^{\frac 32} \| \frac {T\Gamma^{\le m-3} u} {\langle r -t\rangle
}\|_{\infty}
+t^{\frac 12} \| \frac {T \Gamma^{\le m-2} u } {\langle r -t\rangle}
\|_2 \lesssim E_{m-1}^{\frac 12} ;\\
& t^{\frac 12}
\| \langle r-t\rangle \partial^2 \Gamma^{\le m-4} u \|_{\infty}
+ t^{\frac 32}
\| T \partial \Gamma^{\le m-4} u \|_{\infty}
+ t \| T \partial \Gamma^{\le m-4} u \|_2 \lesssim E_{m-1}^{\frac 12}.
\end{align}
These decay estimates play an important role in the nonlinear energy estimates. 

Step 3. Weighted energy estimates: nonlinear terms. We discuss several cases.

Case 1: { $\alpha_1<\alpha$ and $\alpha_2<\alpha$}.
Since $g^{kij}_{\alpha;\alpha_1,\alpha_2}$ still
satisfies the null condition, by Lemma \ref{Lem2.3} we rewrite
\begin{align}
  &\sum_{\substack{\alpha_1<\alpha,  \alpha_2 <\alpha \\ \alpha_1+\alpha_2\le \alpha}}
g^{kij}_{\alpha; \alpha_1,\alpha_2} \partial_k \Gamma^{\alpha_1} u
\partial_{ij} \Gamma^{\alpha_2} u    \notag \\
= &\sum_{\substack{\alpha_1<\alpha,  \alpha_2 <\alpha \\ \alpha_1+\alpha_2\le \alpha}}
g^{kij}_{\alpha;\alpha_1,\alpha_2}
(T_k \Gamma^{\alpha_1} u \partial_{ij} \Gamma^{\alpha_2} u
-\omega_k \partial_t \Gamma^{\alpha_1 } u T_i \partial_j \Gamma^{\alpha_2} u
+ \omega_k \omega_i \partial_t \Gamma^{\alpha_1} u T_j
\partial_t \Gamma^{\alpha_2} u ).
\end{align}

By using the decay estimates obtained in Step 2, we show that
\begin{align} 
  & \sup_{|\alpha|\le m} \| \sum_{\substack{\alpha_1<\alpha,  \alpha_2 <\alpha \\ \alpha_1+\alpha_2\le \alpha}}
g^{kij}_{\alpha;\alpha_1,\alpha_2} \partial_k \Gamma^{\alpha_1} u
\partial_{ij} \Gamma^{\alpha_2} u    \|_2 \lesssim  
t^{-\frac 32} E_{\Flor{\frac {m}2}+3}^{\frac12} E_{m}^{\frac 12}.
\end{align}

Case 2: The quasilinear piece $\alpha_1=0$, $\alpha_2=\alpha$. 
  Recall that $g^{kij}_{\alpha;0,\alpha}
=g^{kij}$. 
By using successive integration by parts, we have
\begin{align}
 \int g^{kij}\pa_{k}u\pa_{ij}\Gamma^{\alpha} u \pa_{t}\Gamma^{\alpha} u e^{p}
= \OK,
\end{align}
where $\OK$ is in the sense of \eqref{No3.5}. Here we exploit an important algebraic
identity (see \eqref{varphi1})
\begin{align}\notag
   &-\pa_{j}\varphi\pa_{i}v\pa_{t}v+\pa_{t}\varphi\pa_{i}v\pa_{j}v-\pa_{i}\varphi\pa_{t}v\pa_{j}v
   \notag\\
   =&-T_{j}\varphi\pa_{i}v\pa_{t}v+\pa_{t}\varphi T_{i}v T_{j} v-T_{i}\varphi\pa_{t}v\pa_{j}v-\omega_{i}\omega_{j}\pa_{t}\varphi(\pa_{t}v)^2,
 \end{align}
where $\varphi$ is taken to be either $\pa_{k}u$ or $e^{p}$, and $v=\Gamma^{\alpha}u$. The standard null form condition amounts
to the annihilation of the term $\omega_i \omega_j \omega_k$ when
$\varphi =\partial_k u$ and  $\partial_t \varphi$ is
replaced by $T_k\partial_t u -\omega_k \partial_{tt} u$.

Case 3: the main piece $\alpha_1=\alpha$, $\alpha_2=0$. By using Lemma \ref{Lem2.3}
with the decay estimates, we derive
\begin{equation}
  \int g^{kij} \pa_{k} \Gamma^{\alpha} u\pa_{ij}u\pa_{t} \Gamma^{\alpha} u e^{p}= \OK+
  \underbrace{\int g^{kij} \omega_{i}\omega_{j}T_{k} \Gamma^{\alpha} u \pa_{tt}u\pa_{t}\Gamma^{\alpha} ue^{p}}_{=:Y_1}.
\end{equation}
We  perform a further refined decomposition of the term $Y_1$.
By using $T_1=\omega_1\partial_+-\frac{\omega_2}r
\partial_{\theta}$
and $T_2=\omega_2\partial_++\frac{\omega_1} r \partial_{\theta}$,   we obtain
(below we denote $v=\Gamma^{\alpha} u$)
\begin{align*}
 g^{kij} \omega_{i}\omega_{j}T_{k}v
 &= g^{1ij} \omega_{i}\omega_{j}(\omega_{1}\pa_{+}v-\frac{\omega_{2}}{r}\pa_{\theta}v)
  +g^{2ij} \omega_{i}\omega_{j}(\omega_{2}\pa_{+}v+\frac{\omega_{1}}{r}\pa_{\theta}v)\\
  &= (g^{1ij} \omega_{1}\omega_{i}\omega_{j}+g^{2ij} \omega_{2}\omega_{i}\omega_{j})\partial_+v +\omega_i\omega_j (g^{2ij}\omega_1
  -g^{1ij} \omega_2) \frac 1 r \partial_{\theta} v \notag \\
  &=: h_1(\theta) \partial_+ v + h_2(\theta) \frac 1 r \partial_{\theta} v.
  \end{align*}
We decompose $Y_1$ accordingly as
\begin{align}
Y_1&= \int h_1(\theta) \partial_+ v \partial_t v \partial_{tt} u e^p 
+ \int h_2(\theta) \frac 1 r \partial_{\theta} v \partial_t v \partial_{tt} u e^p  \notag \\
& =: Y_{A} +Y_B.
\end{align}

Step 4.  Estimate of $Y_A$: localization, further decomposition and normal form transformation.  We use a bump function
$\phi$ which is localized to $r\in [\frac t2,  2t]$ such that the main part of $Y_A$ becomes
\begin{align} \label{e1.19a0}
\int h(\theta) \partial_+ v \partial_t v \partial_{tt}u e^p \phi.
\end{align}
The contribution of the regimes $r\le \frac t2$ and $r>2t$ corresponding to the cut-off $1-\phi$ can be shown to be negligible. 
We further use the decomposition $\partial_t = \frac {\partial_++\partial_-}2$ to transform
\eqref{e1.19a0} as 
\begin{align} \label{1.19A}
\frac 12 \int h(\theta) \partial_+ v \partial_- v \partial_{tt}u e^p \phi +\OK. 
\end{align}
At this point, the crucial observation is to use the fundamental identity 
$
\partial_+ \partial_- =
\square +\frac 1r {\partial_r} +\frac 1 {r^2} {\partial_{\theta\theta}}
$
to transform \eqref{1.19A} into an expression which contains an ``inflated" nonlinearity. After this
novel normal form type transformation and further technical estimates
 the term $Y_A$ can be shown to be under control. 

Step 5. Estimate of $Y_B$: localization and further transformation. By using the 
estimate $\| \langle r -t \rangle \partial_{tt} u \|_{\infty} \lesssim t^{-\frac 12}$ (see Lemma 
\ref{lem2.6}), we have
\begin{align}
Y_B= \OK + \int h_2 (\theta) \frac 1r \partial_{\theta} v \partial_{tt} u \partial_t v
\tilde \phi  (\frac x t) e^p,
\end{align}
where $\tilde \phi$ is a radial bump function localized to $|z|\sim 1$.  Denote $\phi(z) = \frac 1 {|z|} \tilde \phi(z)$. 
Using integration by parts in $\theta$, we obtain
\begin{align}
Y_B= \OK + \frac 1 t \int h_2(\theta) v \partial_{tt} u \partial_t \partial_{\theta} v 
\phi(\frac x t) e^p.
\end{align}
We then proceed to bound the second term (without the $\frac 1t$ factor) as 
\begin{align}
&\int  e^p \phi(\frac x t) h_2(\theta) v \partial_{tt} u \partial_t \partial_{\theta} v  \notag \\
\lesssim &
\Bigl( \| e^p \phi(\frac x t) h_2(\theta) v \partial_{tt} u \|_2 + \| \nabla ( e^p \phi(\frac x t) h_2(\theta) v \partial_{tt} u
)\|_2 \Bigr)\cdot \| \langle \nabla \rangle^{-1} \partial_t \partial_{\theta} v \|_2 \notag \\
\lesssim &\; t^{-\frac 12} E_m^{\frac 12} \| \langle \nabla 
\rangle^{-1} \partial_t \partial_{\theta} v \|_2.
\end{align}
By Lemma \ref{lem_Nonlocal1} the norm $\| \langle \nabla \rangle^{-1} \partial_t \partial_{\theta}
v \|_2$ is well-defined.
The employment of the nonlocal norm $\| \langle \nabla \rangle^{-1} \partial_t \partial_{\theta} v\|_2$
is the key to obtaining sufficient time-decay estimates of $Y_B$.  In
the next step we show $\|\langle \nabla \rangle^{-1} \partial_t \partial_{\theta} v\|_2
\lesssim t^{\delta}$ for some $\delta<\frac 12$ which suffices for time-integrability.

Step 6. Estimate of  $\| \langle  \nabla \rangle^{-1} \partial_t 
\Gamma^{\le m+1} u \|_2$.  This is the most technical part of the proof.  Due to nonlocality 
we work with a frequency localized energy which in the main order is given by
\begin{align}
\tilde E_m =\sum_{J\ge 0} \sum_{|\beta|\le m+1}
2^{-2J} \| e^{\frac p2} \partial P_J \Gamma^{\beta} u \|_2^2,
\end{align}
where $(P_J)_{J\ge 0}$ are the Littlewood-Paley frequency projection operators. 
By using a number of delicate commutator estimates and deeply exploiting the null form
structure, we show $\tilde E_m \lesssim t^{0+}$ which is just enough for closing the uniform
estimates. Here $t^{0+}$ means $t^{c}$ for some sufficiently small exponent $c>0$.

The rest of this paper is organized as follows. In Section 2 we collect some preliminaries and useful
lemmas. In Section 3, 4 and 5 we give the proof of Theorem \ref{thm:main}.

\subsection*{Acknowledgement.}
The author is supported in part by NSFC 12271236. 

\section{Preliminaries}

\subsection*{Notation}
For any two quantities $A$, $B\ge 0$, we write  $A\lesssim B$ if $A\le CB$ for some unimportant constant $C>0$. 
We write $A\sim B$ if $A\lesssim B$ and $B\lesssim A$. We write $A\ll B$ if
$A\le c B$ and $c>0$ is a sufficiently small constant. The needed smallness is clear from the context.

We shall use the Japanese bracket notation: $ \langle x \rangle = \sqrt{1+|x|^2}$, for $x \in \mathbb R^2$.  For $s\in \mathbb R$, we  denote the smoothed fractional Laplacian 
$\langle \nabla \rangle^s=(I-\Delta)^{s/2}$ which corresponds to the Fourier
multiplier $(1+|\xi|^2)^{s/2}$.

We denote  $\partial_0 = \partial_t$, 
$\partial_i = \partial_{x_i}$, $i=1,2$ and (below $\partial_{\theta}$ and
$\partial_r$ correspond to the usual polar coordinates) 
\begin{align}
& \partial = (\partial_i)_{i=0}^2, \; \partial_{\theta}=x_1 \partial_2 -x_2\partial_1,
\; L_0 = t \partial_t + r \partial_r; \\
& \Gamma= (\Gamma_i )_{i=1}^5, \quad\text{where } \Gamma_1 =\partial_t, \Gamma_2=\partial_1,
\Gamma_3= \partial_2, \Gamma_4= \partial_{\theta}, \Gamma_5=L_0; 
\label{def_Gamma}\\
& \Gamma^{\alpha} =\Gamma_1^{\alpha_1} \Gamma_2^{\alpha_2}
\Gamma_3^{\alpha_3} \Gamma_4^{\alpha_4}\Gamma_5^{\alpha_5}, \qquad \text{$\alpha=(\alpha_1,\cdots, \alpha_5)$ is a multi-index}; \label{def_Gamma1}\\
& \partial_+ =\partial_t + \partial_r, \qquad \partial_- =\partial_t - \partial_r; \\
& T_i = \omega_i \partial_t + \partial_i, \; \omega_0=-1, \; \omega_i=x_i/r, \, i=1,2.
\label{2.5a}
\end{align}
Note that in \eqref{def_Gamma} we do not include the Lorentz boosts. Note that $T_0=0$. 
For simplicity of notation,
we define for any integer $k\ge 1$,  $\Gamma^k = (\Gamma^{\alpha})_{|\alpha|=k}$,
$\Gamma^{\le k} =(\Gamma^{\alpha})_{|\alpha|\le k}$. 
In particular
\begin{align}
|\Gamma^{\le k} u | = (\sum_{|\alpha|\le k} |\Gamma^{\alpha} u |^2)^{\frac 12}.
\end{align}
Informally speaking, it is useful to think of  $\Gamma^{\le k} $ as any one of the vector
fields $ \Gamma^{\alpha}$ with $|\alpha| \le k$. 

 For integer $ I\ge 3$, we shall denote 
\begin{align}
E_{I} = E_{I} (u(t,\cdot)) = \| (\partial \Gamma^{\le I} u)(t,\cdot) \|_{L_x^2(\mathbb R^2)}^2.
\end{align}

In Section 5 of this paper we will need to use the Littlewood--Paley (LP) frequency projection
operators. To fix the notation, let $\phi_0$ be a radial function in
$C_c^\infty(\mathbb{R}^2 )$ and satisfy
\begin{equation}\nonumber
0 \leq \phi_0 \leq 1,\quad \phi_0(\xi) = 1\ {\text{ for}}\ |\xi| \leq
1,\quad \phi_0(\xi) = 0\ {\text{ for}}\ |\xi| \geq 7/6.
\end{equation}
Let $\phi(\xi):= \phi_0(\xi) - \phi_0(2\xi)$ which is supported in $\frac 12 \le |\xi| \le \frac 76$.
For any Schwartz function $f \in \mathcal S(\mathbb R^n)$, $j \in \mathbb Z$, define
\begin{align} 
 &\widehat{P_0 f} (\xi) = \phi_0( \xi) \hat f(\xi);  \notag \\
 &\widehat{P_j f} (\xi) = \phi(2^{-j} \xi) \hat f(\xi), \qquad \xi \in \mathbb R^2, j \ge 1; \notag\\
 & \widehat{P_{\le j} f }(\xi)
 = \phi_0( 2^{-j} \xi) \hat f (\xi), \qquad \xi \in \mathbb R^2, \label{LP_def1}
 j\ge 1.
\end{align}
Note that for $j\ge 1$, 
$P_jf $ is supported in the annulus 
$ \frac 12 \cdot 2^j \le |\xi| \le \frac 7 6 \cdot 2^j $.

More generally, one can take  $\psi \in C_c^{\infty}(\mathbb R^2)$ with compact support in $\{ \xi: 
a_1<|\xi| <a_2\}$, and  $0<a_1<a_2<\infty$ are constants. To spell out the explicit dependence
on $\psi$, one can define for $j\ge 0$
\begin{align}
\widehat{P^{\psi}_j  f}(\xi) = \psi( 2^{-j} \xi) \widehat f (\xi).
\end{align}
In this way $P^{\psi}_j$ is a smooth frequency cut-off localized to $|\xi| \sim 2^j$. In later computations
we often write $\tilde P_j = P^{\psi}_j$ where $\psi$ may vary from line to line.  This notation is convenient
for intermediate calculations.

For $j\ge 2$, we will denote $P_{<j} = P_{\le j-1}$,  $P_{>j} = I-P_{\le j}$ ($I$ is the identity operator),
$P_{\ge j} = I - P_{\le j-1}$. 

We begin with the following innocuous lemma which justifies the legitimacy of the norm
$\| \langle \nabla \rangle^{-1} \partial \Gamma^{\le m+1} u \|_2$. 

\begin{lem}[The nonlocal norm is well-defined]  \label{lem_Nonlocal1}
Let $u$ be the solution to \eqref{eq:we2d}. We have
\begin{align}
\| \langle \nabla \rangle^{-1} \partial \Gamma^{\le m+1} u 
\Bigr|_{t=2} \|_{L^2(\mathbb R^2)} \le D_1,
\end{align}
where $D_1>0$ is a finite constant depending on  $\| f_1 \|_{H^{m+1}(\mathbb R^2)}$ and  $\| f_2 \|_{H^m (\mathbb R^2)}$. 
\end{lem}
\begin{proof}
Clearly we only need to consider the case $\| \langle \nabla \rangle^{-1} \partial_t
\Gamma^{\le m+1} u \Bigr|_{t=2} \|_2$. 
Since $\Gamma= \{ \partial_t, \partial_1, \partial_2, x^{\perp}\cdot \nabla, 
t \partial_t + x \cdot \nabla \}$ and $f_1$, $f_2$ are both compactly supported, 
we have
\begin{align}
\| \langle \nabla \rangle^{-1} \partial_t \Gamma^{\le m+1} u 
\Bigr|_{t=2} \|_{L^2(\mathbb R^2)} \lesssim 
\sum_{m_1+m_2+m_3\le m+2}
\| \langle \nabla \rangle^{-1}
\left( \phi_{m_1,m_2,m_3}(x)  \cdot \partial_t^{m_1} \partial_1^{m_2} \partial_2^{m_3} u
\Bigr|_{t=2} \right) \|_{L^2(\mathbb R^2)},
\end{align}
where $\phi_{m_1,m_2,m_3} \in C_c^{\infty}(\mathbb R^2)$.  It is not difficult to check that
for each ($m_1$, $m_2$, $m_3$),  we have
\begin{align}
\partial_t^{m_1} \partial_1^{m_2} \partial_2^{m_3} u
\Bigr|_{t=2} =   F_{m_1,m_2,m_3}^{(0)} + \sum_{j=1}^2 \partial_j
F_{m_1,m_2,m_3}^{(j)},
\end{align}
where
\begin{align}
\sum_{j=0}^2 \| F_{m_1,m_2,m_3}^{(j)} \|_{L^2(\mathbb R^2)} 
\le D_{m_1,m_2,m_3} <\infty,
\end{align}
and $D_{m_1,m_2,m_3}>0$ are constants depending on 
$\|f_1\|_{H^{m+1}(\mathbb R^2)}$ and $\|f_2\|_{H^m(\mathbb R^2)}$. 
The desired result follows.
\end{proof}

 \begin{lem}[Sobolev decay]\label{lem:S}
For $v\in \mathcal S(\R^2)$, we have
\begin{align*}
  \sup_{x\in \mathbb R^2} (|x|^{\frac 12} | v(x)| )\lesssim  \|\pa_{\theta}^{\le1}\pa_{r}^{\le 1} v\|_{2}=\|v \|_{2}+ \|\pa_{r}v\|_{2}
  +\|\pa_{\theta}v\|_{2}+\|\pa_{r}\pa_{\theta}v\|_{2}.
 \end{align*}
  \end{lem}
\begin{proof}
 For a one-variable smooth function $h$ which decays sufficiently fast at the spatial infinity, we have
\begin{equation}\label{Sb-1}
    \rho|h(\rho)|^2 \le \int_0^{\infty} |h(r)|^2 r dr + \int_0^{\infty} |h^{\prime}(r)|^2 r dr, \quad
    \forall\, \rho>0.
\end{equation}
It follows that (below we slightly abuse the notation and denote $v(\rho,\theta)=v(x)$ for
$x=(\rho\cos\theta, \rho \sin\theta)$)
\begin{align*}
  \rho\|\pa_{\theta}v(\rho,\theta)\|_{L^2_\theta}^2
  =\rho\int_{0}^{2\pi}|\pa_{\theta}v(\rho,\theta)|^2 d\theta
  \lesssim\ \|\pa_{\theta}v \|_{L^2(\mathbb R^2)}^2+\|\pa_{r}\pa_{\theta}v \|_{L^2(\mathbb R^2)}^2.
\end{align*}
Denote $\bar v(\rho)$ as the average of $v(\rho,\theta)$ over $\theta$. By \eqref{Sb-1}, we have
\begin{align*}
  \rho|\bar{v}(\rho)|^2
  ={\rho}\left( \frac 1 {2\pi}\int_{0}^{2\pi}v(\rho,\theta) d\theta\right)^2
\lesssim \rho \int_0^{2\pi} v(\rho,\theta)^2 d\theta
  \lesssim \, \|v \|_{L^2(\mathbb R^2) }^2+\|\pa_{r}v \|_{L^2(\mathbb R^2)}^2.
\end{align*}

Note that $|v(\rho,\theta)-\bar{v} (\rho)|^2\lesssim|\pa_{\theta}v|_{L^2_{\theta}}^2$ by the Poincar\'e inequality. Thus
\begin{equation*}
  |x||v(x)|^2
  \lesssim\, \|v \|_{2}^2+ \|\pa_{r}v\|_{2}^2
  +\|\pa_{\theta}v\|_{2}^2+\|\pa_{r}\pa_{\theta}v\|_{2}^2.
\end{equation*}
\end{proof}

\begin{lem}[Refined Hardy's inequality]\label{lem:Hardy}
For any real-valued $h\in C_c^{\infty}([0, M+1))$ with $M>0$, we have
\begin{align} \label{2.6.0a}
\int_0^{M+1} \frac {h(\rho)^2} {(2+M-\rho)^2} \rho d\rho 
\le  \, 4\int_0^{\infty} (h^{\prime}(\rho) )^2 \rho d \rho.
\end{align}

For $u\in C^{\infty}([0,T]\times \R^2)$ with support in $\{(t,x): |x|\leq 1+t\}$, we have
\begin{align*}
 \|\langle |x|-t\rangle^{-1} u \|_{L_x^2(\mathbb R^2) }&\lesssim \|\pa_r u\|_{L_x^2(\mathbb R^2)},
  \qquad     \langle |x|-t\rangle^{-1} |u(t,x)|\lesssim 
  \langle x\rangle^{-\frac 12}  \| \partial \Gamma^{\le 1} u \|_{L_x^2(\mathbb R^2)}. 
\end{align*}

\end{lem}

\begin{proof}
The inequality \eqref{2.6.0a} follows from integrating by parts:
\begin{align}
\text{LHS of \eqref{2.6.0a}}= - \int_0^{M+1}
\frac{h^2} {2+M-\rho} d\rho -\int_0^{M+1} \frac {2hh^{\prime}}{2+M-\rho} \rho d\rho.
\end{align}
The second inequality follows from \eqref{2.6.0a} and the 
fact that $\langle |x|-t\rangle^{-2} \sim (2+t-|x|)^{-2}$ for $|x|\le 1+t$. 
For the third inequality, consider first the case $|x|>1$.  By Lemma \ref{lem:S}, we have
\begin{align}
\langle |x|-t \rangle^{-1} |u(t,x)|
& \lesssim \langle x \rangle^{-\frac 12}  \| \partial_r^{\le 1}
\partial_{\theta}^{\le 1} ( \langle r-t\rangle^{-1} u ) \|_2 \sim
\langle x \rangle^{-\frac 12} \| \partial_r^{\le 1} ( \langle r -t \rangle^{-1} \partial_{\theta}^{\le 1} 
u) \|_2 \notag \\
& \lesssim \langle x \rangle^{-\frac 12} \| \partial_r \partial_{\theta}^{\le 1} u\|_2 
\lesssim \langle x \rangle^{-\frac 12}
 \| \partial \Gamma^{\le 1} u \|_2.
\end{align}
On the other hand, for $|x|\le 1$, we have
\begin{align*}
\langle |x| - t \rangle^{-1} |u(t,x)| 
 & \lesssim \langle t \rangle^{-1}  ( \| u \|_{L^2_x(|x| \le 1)} + \| \partial^2 u \|_{L_x^2(|x|\le 1)} )  \notag \\
 & \lesssim \| \langle |x|- t\rangle^{-1} u \|_{L_x^2(\mathbb R^2)}+ \| \Delta u \|_{L_x^2(\mathbb R^2)} \lesssim \| \nabla u \|_2
 +\|\Delta u\|_2.
 \end{align*}
\end{proof}

\begin{lem} \label{Lem2.3}
If $g^{kij}$ satisfies the null condition, then for $t> 0$  we have
\begin{align} \label{2.10A}
g^{kij} \partial_k f
\partial_{ij} h   = g^{kij}
(T_k  f \partial_{ij} h
-\omega_k \partial_t f T_i \partial_j h
+ \omega_k \omega_i \partial_t f T_j
\partial_t h),
\end{align}
where $T=(T_1,T_2)$ is defined in \eqref{2.5a}.  It follows that
\begin{align}
|g^{kij} \partial_k f \partial_{ij} h| 
& \lesssim | T f | |\partial^2 h| + |\partial f | | T \partial h|  \label{a2.12a}\\
& \lesssim 
\frac 1 {\langle r +t \rangle} 
(|\Gamma f| |\partial^2 h| + |\partial f | |\Gamma \partial h|
+ |\partial f | \cdot |\partial^2 h| \cdot |r -t | ). \label{a2.12b}
\end{align}
Suppose $g^{kij}$ satisfies the null condition and 
$
\square u = g^{kij} \partial_k u \partial_{ij} u.
$
Then for any multi-index $\alpha$, we have
\begin{align} \label{a2.12aa}
\square \Gamma^{\alpha} u  = \sum_{\alpha_1+ \alpha_2 \le \alpha}
g^{kij}_{\alpha;\alpha_1,\alpha_2} \partial_k \Gamma^{\alpha_1} u 
\partial_{ij} \Gamma^{\alpha_2} u,
\end{align}
where for each ($\alpha$, $\alpha_1$, $\alpha_2$), $g^{kij}_{\alpha;\alpha_1,\alpha_2}$ also satisfies the
null condition.  In addition, we have $g^{kij}_{\alpha;\alpha,0} =g^{kij}_{\alpha;0,\alpha}=g^{kij}$.
\end{lem}
\begin{proof}
The identity \eqref{2.10A} follows by applying repeatedly the identity $\partial_l
=T_l - \omega_l \partial_t$ and using the null condition at the last step. 
The inequality \eqref{a2.12b} is obvious if $r\le \frac t2$ or $r\ge 2t$, or $r \sim t \lesssim 1$ since
$\langle r +t \rangle \sim \langle r-t \rangle $ in these regimes. On the other hand, if
$r\sim t\gtrsim 1$,  then one can use the identities
\begin{align}
&T_1 = \omega_1 \partial_+ -\frac {\omega_2} r \partial_\theta, \quad
T_2 = \omega_2 \partial_+ +\frac {\omega_1} r \partial_\theta ;\quad
 \partial_+ = \frac 1 {t+r} ( 2 L_0 - (t-r) \partial_-).
\end{align}
The identity \eqref{a2.12aa} follows from H\"ormander \cite{Hor97}.
\end{proof}

\begin{lem} \label{lem2.5A}
For any $f \in \mathcal S(\mathbb R^2)$, we have
\begin{align} 
&\sup_{x_0 \in \mathbb R^2} \langle |x_0| -t \rangle^{\frac 12} |f(x_0)| \lesssim \|f\|_2+\| \langle |x|-t\rangle \nabla f
\|_2 + \| \langle |x|-t\rangle \partial_1 \partial_2 f\|_2, 
\qquad \forall\, t\ge 0;  \label{2.26A}\\
& \| \langle |x| -t \rangle \partial f \|_{\infty} 
\lesssim \| \langle |x|-t\rangle \partial f \|_2 + \| \langle |x| -t\rangle \partial^2 f \|_2
+ \| \langle |x| -t \rangle \partial^3 f \|_2, \, \quad\forall\, t\ge 0.
\label{2.26B}
\end{align}
It follows that
\begin{align} 
&\| f \|_{L_x^{\infty}(\mathbb R^2)} \lesssim
\langle t \rangle^{-\frac 12}
(\|f\|_2 + \| \partial_{\theta} f\|_2+ \| \langle |x| -t\rangle \nabla \tilde \Gamma^{\le 1 } f \|_2), \qquad\forall\, t\ge 0, 
\label{2.27A}
\end{align}
where $\tilde \Gamma =(\partial_1, \partial_2, \partial_{\theta})$. 
\end{lem}
\begin{proof}
The case $||x_0|-t|\le 2$ follows from the inequality
$|f(x_0)|^2 \le \int|\partial_1\partial_2(f(x)^2) | dx_1 dx_2$. For $||x_0|-t|>2$,
we note that $\langle x_0\rangle +t \sim |x_0|+t\gtrsim 1$ and
$\langle |x_0|-t \rangle \sim \frac{ \langle |x_0|^2-t^2 \rangle}{\langle x_0
\rangle + t}=:W(x_0)$.   Observe that
\begin{align}
\sum_{1\le i\le 2}\|\partial_i  W\|_{\infty}+ \sum_{1\le i,j\le 2}\|\partial_i \partial_j W\|_{\infty} \lesssim 1, \quad W(x) \lesssim \langle |x|-t\rangle, \;\forall\, x \in \mathbb R^2.
\end{align}
 By using the Fundamental Theorem of Calculus we have 
\begin{align}
 \langle |x_0|-t \rangle | f(x_0)|^2 &\lesssim  W(x_0) |f(x_0)|^2 \lesssim \int_{\mathbb R^2}
\left|\partial_1\partial_2 \Bigl ( W(x) f(x)^2
\Bigr) \right | dx_1 dx_2 \notag \\
& \lesssim \| f\|_2^2 + \| \nabla f \|_2 \| f\|_2 + \|
 \langle |x|-t\rangle \nabla f \|_2 \| \nabla f\|_2  + \| \langle |x| -t\rangle \partial_1\partial_2 f \|_2
 \| f\|_2\notag \\
 & \lesssim \|f\|_2^2 + \| \langle |x|-t\rangle \nabla f \|_2^2+ \|
 \langle |x|-t\rangle \partial_1 \partial_2 f \|_2^2.
\end{align}
Thus \eqref{2.26A} follows. The proof of \eqref{2.26B} is similar by working with the expression $W(x_0)^2 |\partial f(x_0)|^2$ for the case $|x_0-t|>2$.
 For \eqref{2.27A} we may assume $t\ge 2$. The
case $|x_0|\le t/2$ follows from \eqref{2.26A}. The case $|x_0|>t/2$ follows from
Lemma \ref{lem:S}. 
\end{proof}

\begin{lem} \label{lem2.3a}
Suppose $\tilde u= \tilde u(t,x)$ has continuous second order derivatives. Then
\begin{align} 
&| \langle r -t \rangle \partial_{tt} \tilde u (t,x) | 
+| \langle r -t \rangle \partial_t \nabla  \tilde u (t,x) | 
+| \langle r -t \rangle \Delta \tilde u (t,x) |  \notag \\
\lesssim &
| (\partial \Gamma^{\le 1} \tilde u)(t,x)| + (r+t) | (\square \tilde u)(t,x) |, \quad r=|x|, \, 
t\ge 0;  \label{2.9a0}
\end{align}
and 
\begin{align}
&| \langle r -t \rangle  \partial^2 \tilde u (t,x) | 
\lesssim 
| (\partial \Gamma^{\le 1} \tilde u)(t,x)| + (r+t) | (\square \tilde u)(t,x) |, \quad 
\forall\, r\ge t/10,
 \, t\ge 1. \label{2.9a1}
\end{align}
Suppose $T_0\ge 2$ and $u \in C^{\infty}([2,T_0]\times \mathbb R^2)$ solves  \eqref{eq:we2d} with support in $|x|\le t+1$, $2\le t\le T_0$.
For any integer $l_0\ge 2$, there exists $\epsilon_1>0$ depending only on
$l_0$, such that if at some $2\le t\le T_0$,
\begin{align} \label{2.9a3}
\| (\partial \Gamma^{\le \lceil{\frac {l_0}2}\rceil +2} u)(t,\cdot) \|_{L_x^2(\mathbb R^2)} \le \epsilon_1,
\qquad  \text{( here $\lceil{z}\rceil = \min \{n\in \mathbb N:\, n\ge z\}$ )}
\end{align}
then for the same $t$, we have the $L^2$ estimate:
\begin{align} \label{2.9a4}
\| (\langle r -t \rangle \partial^2 \Gamma^{\le l_0} u) (t,\cdot) \|_{L_x^2(\mathbb R^2)}\lesssim \| (\partial \Gamma^{\le l_0+1} u )(t,\cdot) \|_{L_x^2(\mathbb R^2)}.
\end{align}
For any integer $l_1\ge 2$, there exists $\epsilon_2>0$ depending only on
$l_1$, such that if at some $2\le t\le T_0$,
\begin{align} \label{2.99a1}
\| (\partial \Gamma^{\le l_1 +1} u)(t,\cdot) \|_{L_x^2(\mathbb R^2)} \le \epsilon_2,
\end{align}
then for the same $t$, we have the point-wise estimate:
\begin{align} \label{2.99a2}
 |(\langle r -t  \rangle \partial^2 \Gamma^{\le l_1} u )(t,x) |
 \lesssim |  (\partial \Gamma^{\le l_1+1} u )(t,x)|, 
 \qquad\forall\, r\ge t/10.
 \end{align}
 Moreover, we have
 \begin{align} \label{2.99aB2}
 \|  \partial \Delta \Gamma^{\le l_1-1} u \|_{L_x^2(|x|\le \frac 23 t)} 
 \lesssim t^{-2}\| (\partial \Gamma^{\le l_1 +1} u)(t,\cdot) \|_{L_x^2(\mathbb R^2)}.
 \end{align}
 \end{lem}
\begin{proof}
In the 3D case, the estimate \eqref{2.9a0} is an elementary but deep observation of Sideris (cf. \cite{SK96}).  To prove the 2D case we denote 
$Y= |(\partial \Gamma^{\le 1} \tilde u)(t,x)| + (r+t) | (\square \tilde u)(t,x) |$. By using
$L_0 \tilde u =t \partial_t \tilde u +r \partial_r \tilde u$, we obtain
\begin{align}
& \partial_t L_0 \tilde u = \partial_t \tilde u
+t \partial_{tt} \tilde u + r \partial_t \partial_r \tilde u
\; \Rightarrow \; | r \partial_t \partial_r \tilde u + t \partial_{tt} \tilde u| \lesssim Y; \\
& \partial_r L_0 \tilde u = t \partial_t \partial_r \tilde u
+ \partial_r \tilde u +r \partial_{rr} \tilde u
\; \Rightarrow \; | t \partial_t \partial_r \tilde u + r \partial_{rr} \tilde u| \lesssim Y.
\end{align}
Since $\square \tilde u = \partial_{tt} \tilde u -\partial_{rr} \tilde u -
\frac 1 r \partial_r \tilde u - \frac 1 {r^2}
\partial_{\theta\theta} \tilde u$ and $|\frac 1 r \partial_{\theta\theta} \tilde u|
\lesssim | \nabla \partial_{\theta} \tilde u| \lesssim Y$, we have
\begin{align}
|r \partial_{tt} \tilde u - r \partial_{rr} \tilde u| \lesssim Y.
\end{align}
It follows that 
\begin{align}
\langle r -t\rangle ( |\partial_{tt} \tilde u| + |\partial_t \partial_r \tilde u|) \lesssim Y.
\end{align}
By using $\partial_{\theta} L_0 \tilde u
=t \partial_t \partial_{\theta} \tilde u + r \partial_r \partial_{\theta}
\tilde u=(t-r) \partial_t \partial_{\theta} \tilde u + r (\partial_t +\partial_r )\partial_{\theta}
\tilde u$,  we obtain
\begin{align}
|(t-r) \partial_t ( \frac 1 r \partial_{\theta} \tilde u) | \lesssim Y.
\end{align}
The estimates of $\partial_t \partial_r \tilde u$ and 
$\partial_t (\frac 1 r \partial_{\theta} \tilde u)$ settle the point-wise estimate of
$\partial_t \nabla \tilde u$. It follows that
\begin{align}
\langle r -t\rangle ( |\partial_{tt} \tilde u| + |\partial_t \nabla \tilde u|+|\Delta \tilde u|)
\lesssim Y
\end{align}
which is exactly \eqref{2.9a0}.  To derive the estimate \eqref{2.9a1} 
we only need to bound $|(r-t) \partial_i \partial_j \tilde u|$ for $1\le i, j\le 2$. 
By using the identity $\nabla = \omega \partial_r + 
\frac {\omega^{\perp} } r \partial_{\theta}$ where $\omega=(\cos \theta,
\sin \theta)$,
$ \omega^{\perp}=(-\sin \theta, \cos \theta)$, it is not difficult to check that
for $r\gtrsim t$, 
\begin{align}
\sum_{1\le i,j \le 2}
|(r-t) \partial_i \partial_j \tilde u|
&\lesssim |(r-t) \partial_{rr} \tilde u| + Y \notag \\
& \lesssim | (r-t) (\Delta \tilde u - \frac 1 r \partial_r \tilde u -\frac 1 {r^2} \partial_{\theta\theta}
\tilde u) | + Y  \lesssim Y.
\end{align}
Thus \eqref{2.9a1} easily follows.

 For \eqref{2.9a4}, by using a simple integration-by-parts argument, one has (below 
 $k_0\ge 0$ is a  running parameter)
 \begin{align} \label{2.21A}
\sum_{i,j=1}^2 \| \langle r -t\rangle \partial_i \partial_j \Gamma^{\le k_0} u \|_2
  \lesssim \| \partial \Gamma^{\le k_0} u \|_2 + \| \langle r -t\rangle \Delta 
  \Gamma^{\le k_0} u\|_2.
  \end{align}
 By using \eqref{2.9a0} and
\eqref{a2.12b} we have
\begin{align}
 & |(\langle r -t \rangle \partial_{tt} \Gamma^{\le k_0} u)(t,x)|  +
 |(\langle r -t \rangle \partial_t \nabla  \Gamma^{\le k_0} u)(t,x)|  +|(\langle r -t \rangle \Delta \Gamma^{\le k_0} u)(t,x)|   \notag \\
\lesssim  &
| (\partial \Gamma^{\le k_0+1}  u)(t,x)| 
+ \sum_{m+l\le k_0}
(|\Gamma^{\le m+1} u|
|\partial^2 \Gamma^{\le l} u |
+|\partial \Gamma^{\le m} u| |\Gamma^{\le l+1} \partial u|+
|\partial \Gamma^{\le m} u | |\partial^2 \Gamma^{\le l} u| |r-t| ). \label{2.23P}
\end{align}
By \eqref{2.21A}, we obtain
\begin{align}
&\| \langle r-t\rangle \partial^2 \Gamma^{\le k_0} u \|_2 
\lesssim  
\| \partial \Gamma^{\le k_0+1}  u\|_2  \notag \\
&\qquad+ \sum_{m+l\le k_0}
(\| |\Gamma^{\le m+1} u|
|\partial^2 \Gamma^{\le l} u | \|_2
+\| |\partial \Gamma^{\le m} u| |\Gamma^{\le l+1} \partial u| \|_2+
\| |\partial \Gamma^{\le m} u | |\partial^2 \Gamma^{\le l} u| |r-t| \|_2).
\end{align}
If $m\le l+1$, then we use  the estimates 
(note that $m+2 \le \lfloor
\frac{k_0+1}2 \rfloor +2\le \lceil \frac {k_0}2\rceil+2$)
\begin{align}
&\langle r -t\rangle^{-1}| (\Gamma^{\le m+1} u)(t,x)| \lesssim
\| \partial \Gamma^{\le m+2} u \|_2, 
\quad \| \partial \Gamma^{\le m} u\|_{\infty}
\lesssim \| \partial \Gamma^{\le m+2} u\|_2.
\end{align}
If $m\ge l+2 $, then $l\le \frac{k_0-2}2$ and we use the estimates 
(see \eqref{2.26B} for the second estimate)
\begin{align} \label{2.25A}
\| \frac{|\Gamma^{\le m+1} u|}{\langle r-t\rangle} \|_2 \lesssim
\| \partial \Gamma^{\le m+1} u\|_2, \quad
| \langle r -t \rangle \partial^2 \Gamma^{\le l} u(t,x)|
\lesssim \|\langle r -t\rangle \partial^2 \Gamma^{\le l+2} u\|_2.
\end{align}
Thus if $\| \partial\Gamma^{\le k_0+1} u \|_2 \ll 1$, we obtain 
\begin{align}
\| \langle r-t \rangle \partial^2 \Gamma^{\le k_0} u(t,\cdot) \|_2
\lesssim \| \partial \Gamma^{\le k_0+1} u\|_2 \ll 1.
\end{align}
To prove \eqref{2.9a4} under the assumption \eqref{2.9a3} we first take $k_0=
\lceil{\frac {l_0}2}\rceil+1 $ and show that 
\begin{align}
\|\langle r-t\rangle \partial^2 \Gamma^{\le \lceil{\frac {l_0}2}\rceil +1} u)(t,\cdot) \|_2 
\lesssim\| (\partial \Gamma^{\le \lceil{\frac {l_0}2}\rceil +2} u)(t,\cdot) \|_2 \ll 1.
\end{align}
We then use this smallness in \eqref{2.25A} and obtain the desired result for
$k_0= l_0$ (Note that $\lceil \frac{l_0-2} 2 \rceil +2 \le \lceil \frac{l_0} 2 \rceil+1$). 
The estimate of \eqref{2.99a2} follows from  \eqref{2.9a1}.

We  turn now to \eqref{2.99aB2}.  Applying \eqref{2.9a0} to $\tilde u
= \partial \Gamma^{\le l_1-1} u$ with $r\le \frac 23 t$, we get
\begin{align}
| \Delta \partial \Gamma^{\le l_1-1} u |
\lesssim \frac 1 t |\partial^2 \Gamma^{\le l_1 } u|
+ | \partial \square \Gamma^{\le l_1-1} u|.
\end{align}
By Lemma \ref{Lem2.3}, we have
\begin{align}
| \partial \square \Gamma^{\le l_1-1} u|
& \lesssim \sum_{a+b\le l_1-1} | \partial( \partial \Gamma^a u \partial^2 \Gamma^b u) |
\notag \\
& \lesssim 
|\partial^2 \Gamma^{\le l_1-1} u| | \partial^2 \Gamma^{\le l_1-1} u| +|\partial \Gamma^{\le l_1-1}u|
|\partial^3 \Gamma^{\le l_1-1} u|. \label{A2.21_0}
\end{align}
Note that 
\begin{align}
|\partial^3 \Gamma^{\le l_1-1} u|
&\lesssim | \underbrace{\partial_{tt} \partial \Gamma^{\le l_1-1} u }_{\text{$\partial_t$ appears
twice or more}} |+ \sum_{1\le i_1,i_2\le 2} |\underbrace{
\partial \partial_{i_1} \partial_{i_2} \Gamma^{\le l_1-1} u}_{\text{$\partial_t$
appears at most once}}|  \notag \\
& \lesssim | \partial \square \Gamma^{\le l_1-1} u|
+ |\partial \tilde \partial^2 \Gamma^{\le l_1-1} u |, \label{A2.21_1}
\end{align}
where we have denoted $\tilde \partial = (\partial_1, \partial_2)$.  By using the smallness
of the pre-factor $\|\partial \Gamma^{\le l_1-1} u\|_{\infty}$ and
\eqref{A2.21_1},  we then derive from \eqref{A2.21_0}
\begin{align}
|\partial \square \Gamma^{\le l_1-1} u |\lesssim 
|\partial^2 \Gamma^{\le l_1-1} u| | \partial^2 \Gamma^{\le l_1-1} u| +|\partial \Gamma^{\le l_1-1}u|
|\partial \tilde \partial^2 \Gamma^{\le l_1-1} u|. \label{A2.21_2}
\end{align}

By the standard Sobolev embedding $H^1(\mathbb R^2) \hookrightarrow L^4(\mathbb R^2)$, 
we have
\begin{align}
&\| \langle r-t\rangle \partial^2 \Gamma^{\le l_1-1} u
\|_4  \lesssim \| \partial^{\le 1} ( \langle r -t \rangle  \partial^2 \Gamma^{\le l_1-1} u)
\|_2 \lesssim \| (\partial \Gamma^{\le l_1 +1} u)(t,\cdot) \|_2;  \\
& \| \langle r -t \rangle^{\frac 12}  \partial \Gamma^{\le l_1-1} u\|_{\infty}
 \lesssim \| (\partial \Gamma^{\le l_1 +1} u)(t,\cdot) \|_2,
\qquad (\text{by Lemma \ref{lem2.5A}}). 
\end{align}
By using a smooth cut-off function localized to $|x| \le \frac 23 t$, we then derive
\begin{align}
\| \Delta \partial \Gamma^{\le l_1-1} u\|_{L_x^2(|x|\le \frac 23 t)} \lesssim t^{-\frac 32} \| (\partial \Gamma^{\le l_1 +1} u)(t,\cdot) \|_2.
\end{align}
It follows that (recall $\tilde \partial=(\partial_1, \partial_2)$)
\begin{align}
\| \tilde \partial^2 \partial \Gamma^{\le l_1-1} u\|_{L_x^2(|x|\le \frac 23 t)} \lesssim t^{-\frac 32}\| (\partial \Gamma^{\le l_1 +1} u)(t,\cdot) \|_2.
\end{align}
Plugging this estimate into \eqref{A2.21_2},  we 
obtain the estimate \eqref{2.99aB2}.
\end{proof}

\begin{lem}[Decay estimates] \label{lem2.6}
Suppose $T_0\ge 2$ and $u \in C^{\infty}([2,T_0]\times \mathbb R^2)$ solves  \eqref{eq:we2d} with support in $|x|\le t+1$, $2\le t\le T_0$. Suppose $I\ge 4$ is an integer  and 
\begin{align}
E_I(u(t,\cdot) ) = \| (\partial \Gamma^{\le I} u )(t,\cdot) \|_2^2 \le \tilde \epsilon,
\end{align}
where $\tilde \epsilon>0$ is sufficiently small. Then we have the following decay estimates:
\begin{align}
& t^{\frac 12} \| \partial \Gamma^{\le I-2} u \|_{L_x^{\infty}} 
+t^{\frac 12}\| \langle |x|-t\rangle \partial^2 \Gamma^{\le I-3} u \|_{L_x^{\infty}(|x|>\frac{t}{10})} 
+\| \langle |x|-t\rangle \partial^2 \Gamma^{\le I-3} u \|_{L_x^{\infty}}
\lesssim \;E_I^{\frac 12}; \label{2.30A}\\
& \| \partial^2\Gamma^{\le I-3} u \|_{L_x^{\infty}(|x|<\frac t2)}
\lesssim t^{-\frac 32} E_I^{\frac 12}; \qquad \label{2.30AA}\\
& 
\| \langle |x|-t\rangle \partial^2 \Gamma^{\le I-3} u \|_{L_x^{\infty}}
\lesssim t^{-\frac 12} E_I^{\frac 12} ;
\qquad \label{2.30AB}\\
&\| \frac{T \Gamma^{\le I-2} u} {\langle |x|-t \rangle} \|_{L_x^{\infty}}
+ \| T \partial \Gamma^{\le I-3} u \|_{L_x^{\infty}}
\lesssim t^{-\frac 32} E_I^{\frac 12};  \label{2.30B} \\
& \| \frac{ T \Gamma^{\le I-1} u } { \langle |x| -t \rangle }
\|_{L_x^2}  +\| T \partial \Gamma^{\le I-1} u \|_{L_x^2} \lesssim
t^{-1} E_I^{\frac 12}. \label{2.30C}  
\end{align}
More generally, for any integer $I_1\ge 1$, we have
\begin{align} \label{2.30D}
\| \frac{T\Gamma^{\le I_1} u } {\langle |x| -t \rangle }
\|_{L_x^2} \lesssim t^{-1} \| \partial \Gamma^{\le I_1+1} u \|_{L_x^2}.
\end{align}
Also we have
\begin{align}
& \| \langle |x|-t\rangle^2 \partial^3 u \|_{L_x^{\infty}(|x|\ge \frac t2)}
\lesssim t^{-\frac 12}. \label{2.30DD}
\end{align}
\end{lem}
\begin{proof}
We shall take $\tilde \epsilon$ sufficiently small so that Lemma
\ref{lem2.3a} can be applied.   The estimate \eqref{2.30A} follows from Lemma \ref{lem2.5A} and Lemma \ref{lem2.3a}.
To derive the estimate \eqref{2.30AA}, we choose $\psi \in C_c^{\infty}(\mathbb R^2)$
such that $\psi(z) \equiv 1$ for $|z|\le 0.5$ and $\psi(z) \equiv 0$ for $|z|\ge 0.52$. Applying
the interpolation inequality $\|\tilde v\|_{\infty} \lesssim \|\tilde v \|_2^{\frac 12}
 \|\Delta \tilde v \|_2^{\frac 12}$ with $\tilde v(x) =\psi(\frac x t) \partial^2 
 \Gamma^{\le J-3} u$, we obtain
 \begin{align}
 \| \psi(\frac x t) \partial^2 \Gamma^{\le J-3} u 
 \|_{\infty} \lesssim \| \psi(\frac x t) \partial^2 \Gamma^{\le I-3} u 
 \|_2^{\frac 12}
 \| \Delta ( \psi(\frac x t) \partial^2 \Gamma^{\le I-3} u  ) \|_2^{\frac 12}.
 \end{align}
 By Lemma \ref{lem2.3a}, it is not difficult to check that
\begin{align}
 \| \Delta ( \psi(\frac x t) \partial^2 \Gamma^{\le I-3} u  ) \|_2 \lesssim t^{-2} E_I^{\frac 12},
 \quad  \| \psi(\frac x t) \partial^2 \Gamma^{\le I-3} u   \|_2  \lesssim t^{-1}
 E_I^{\frac 12}.
 \end{align}
 The estimate \eqref{2.30AA} then follows. For the estimate \eqref{2.30AB} we only
 need to examine the regime $|x| \ge t/2$. But this follows from Lemma \ref{lem2.3a} and
 \ref{lem2.5A}.

For \eqref{2.30B}, we note that the case $|x|\le \frac t2$ follows from \eqref{2.30A} and 
\eqref{2.30AA}.
On the other hand, for $|x|>\frac t2$ we denote $\tilde u=\Gamma^{\le I-2} u$ and estimate
$\| \frac {T_1 \tilde u} {\langle |x|-t \rangle} \|_{L_x^{\infty}(|x|>\frac t2)}$ (the
estimate for $T_2$ is similar). Recall that
\begin{align}
T_1 \tilde u & = \omega_1 \partial_t \tilde u + \partial_1 \tilde u
= \omega_1 (\partial_t + \partial_r ) \tilde u- \frac{\omega_2} r \partial_{\theta} \tilde u\notag \\
& = \omega_1 \frac 1 {t+r} ( 2 L_0 \tilde u - (t-r) \partial_-\tilde u) -\frac{\omega_2} r \partial_{\theta}
\tilde u.
\end{align}
Clearly for $r=|x|\ge \frac t2$,
\begin{align}
\left| \frac{ T_1 \tilde u } { \langle r -t \rangle}
\right| &\lesssim \frac1 t \Bigl( \left| \frac{L_0 \tilde u} {\langle r -t\rangle } \right| + |\partial \tilde u|
\Bigr)
+ \left| \frac{\partial_{\theta} \tilde u} {r \langle r-t \rangle } \right| \notag \\
& \lesssim t^{-1} \cdot t^{-\frac 12} \| \partial \Gamma^{\le 2} \tilde u\|_2
+ t^{-\frac 32}+t^{-1} \cdot t^{-\frac 12} \| \partial \Gamma^{\le 1} \partial_{\theta} \tilde u\|_2 \lesssim
t^{-\frac 32} E_I^{\frac 12},
\end{align}
where in the second last step we used Lemma \ref{lem:Hardy} (for the term $|\partial \tilde u|$
we use \eqref{2.30A}).  The estimates for \eqref{2.30C}--\eqref{2.30D} is similar. We omit the details.
 We now sketch how to
prove \eqref{2.30DD}.  By using \eqref{2.9a1} (applied to $\tilde u =\partial u$),
we obtain
\begin{align}
|\langle r -t\rangle \partial^3 u|
\lesssim | \partial^2 \Gamma^{\le 1} u |
+ (r+t) | \square \partial u |.
\end{align}
The contribution of the term $|\partial^2 \Gamma^{\le 1} u|$ is clearly OK for us since it can
absorb a factor of $\langle r -t \rangle$. By Lemma \ref{Lem2.3} (with $f=\partial u$,
$h=u$ or $f=u$, $h=\partial u$), we have
\begin{align}
(r+t) |\square \partial u |
&\lesssim
|\Gamma \partial u | |\partial^2 u | + |\partial^2 u |^2 |r-t|
+|\Gamma u| |\partial^3 u | + |\partial u | |\Gamma \partial^2 u|
+ |\partial u | |\partial^3 u | |r-t| \notag \\
& \lesssim |\partial \Gamma^{\le 1} u |
| \partial^2 \Gamma^{\le 1} u|+|\partial^2 u |^2 |r-t|
+|\frac{\Gamma u}{\langle r-t\rangle}| \cdot  \langle r-t\rangle |\partial^3 u |
+ |\partial u | |\partial^3 u | |r-t|.
\end{align}
The desired estimate  clearly follows by using smallness of the pre-factors.
\end{proof}

\section{Proof of Theorem \ref{thm:main} }
In this section and later sections, we carry out the proof of Theorem \ref{thm:main}. 
We fix a multi-index $\alpha$ with $|\alpha|\le m$ and for simplicity denote $v= \Gamma^{\alpha} u$. By Lemma \ref{Lem2.3} we have
(below for simplicity of notation we write $g^{kij}_{\alpha_1,\alpha_2}=
g^{kij}_{\alpha;\alpha_1,\alpha_2}$)
\begin{align} 
\square v  &= \sum_{\alpha_1+ \alpha_2 \le \alpha}
g^{kij}_{\alpha_1,\alpha_2} \partial_k \Gamma^{\alpha_1} u 
\partial_{ij} \Gamma^{\alpha_2} u \label{3.0A}\\
&= 
g^{kij}\partial_k v
\partial_{ij} u  + 
g^{kij} \partial_k  u 
\partial_{ij} v 
+\sum_{\substack{\alpha_1<\alpha, \alpha_2<\alpha;\\ \alpha_1+ \alpha_2 \le \alpha}}
g^{kij}_{\alpha_1,\alpha_2} \partial_k \Gamma^{\alpha_1} u 
\partial_{ij} \Gamma^{\alpha_2} u. 
\end{align}

Choose $p(t,r) = q(r-t)$, where
\begin{align}
q(s) = \int_0^s \langle \tau\rangle^{-1}  \bigl(\log ( 2+\tau^2) \bigr)^{-2} d\tau, \quad s\in \mathbb R.
\end{align}
Clearly 
\begin{align}
-\partial_t p = \partial_r p = q^{\prime}(r-t)
= \langle r -t \rangle^{-1} \bigl( \log (2+(r-t)^2)  \bigr)^{-2}.
\end{align}
Multiplying both sides of \eqref{3.0A} by $e^p \partial_t v$, we obtain
\begin{align*}
   \text{LHS}&=\int e^p \pa_{tt}v\pa_{t}v -\int e^p \Delta v\pa_{t}v
   =\int e^p \pa_{tt} v\pa_{t}v +\int e^p\na v\cdot\na\pa_{t}v+\int e^p\na v\cdot\na p\pa_{t}v \\
   &=\frac12\frac{d}{dt}\int e^p(\pa v)^2 -\frac12\int e^{p}|\pa v|^2 p_{t}+\int e^p\na v\cdot\na p\pa_{t}v \\
   &=\frac12\frac{d}{dt}\| e^\frac{p}{2}\pa v\|_{L^2}^2 +\frac 12 
   \int e^p q^{\prime} \cdot \Bigl( |\partial_+ v|^2+ \frac {|\partial_{\theta} v|^2}{r^2}  \Bigr)
   =\frac12\frac{d}{dt}\| e^\frac{p}{2}\pa v\|_{L^2}^2 +\frac 12 
   \int e^p q^{\prime} |T v|^2.
\end{align*}

To simplify the notation in the subsequent nonlinear estimates, we introduce the following
terminology.

\noindent
\textbf{Notation}.  For a quantity $X(t)$, we shall write $X(t)= \mathrm{OK}$  if $X(t)$ can be written as
\begin{align} \label{No3.5}
X(t)= \frac d {dt} X_1 (t)+X_2(t) +X_3(t), 
\end{align}
where  (below $\alpha_0>0$ is some constant)
\begin{align}
|X_1(t)| \ll \| (\partial \Gamma^{\le m}  u)(t,\cdot) \|_{L_x^2(\mathbb R^2)}^2, 
\quad |X_2(t)| \ll    \sum_{|\alpha| \le m} \int e^p q^{\prime} |(T \Gamma^{\alpha} u)(t,x) |^2 dx,
\quad |X_3(t)| \lesssim \langle t \rangle^{-1-\alpha_0}.
\end{align}
In yet other words, the quantity $X$ will be controllable  if either it can be absorbed into
the energy, or can be controlled by the weighted $L^2$-norm of the good
unknowns  from the Alinhac weight,  or it is integrable in time.

We now proceed with the nonlinear estimates. We shall discuss several cases.

\subsection{The case $\alpha_1<\alpha$ and $\alpha_2<\alpha$}
Since $g^{kij}_{\alpha_1,\alpha_2}$ still
satisfies the null condition, by \eqref{2.10A} we have
\begin{align}
  &\sum_{\substack{\alpha_1<\alpha,  \alpha_2 <\alpha \\ \alpha_1+\alpha_2\le \alpha}}
g^{kij}_{\alpha_1,\alpha_2} \partial_k \Gamma^{\alpha_1} u 
\partial_{ij} \Gamma^{\alpha_2} u    \notag \\
= &\sum_{\substack{\alpha_1<\alpha,  \alpha_2 <\alpha \\ \alpha_1+\alpha_2\le \alpha}}
g^{kij}_{\alpha_1,\alpha_2}
(T_k \Gamma^{\alpha_1} u \partial_{ij} \Gamma^{\alpha_2} u
-\omega_k \partial_t \Gamma^{\alpha_1 } u T_i \partial_j \Gamma^{\alpha_2} u
+ \omega_k \omega_i \partial_t \Gamma^{\alpha_1} u T_j
\partial_t \Gamma^{\alpha_2} u ). 
\end{align}

\texttt{Estimate of $\| T_k \Gamma^{\alpha_1} u  \partial^2 \Gamma^{\alpha_2} u\|_2$}.
If $|\alpha_1|\le |\alpha_2|$,  then by Lemma \ref{lem2.6} we have
\begin{align}
 \| T_k \Gamma^{\alpha_1} u  \partial^2 \Gamma^{\alpha_2} u\|_2
 \lesssim
 \| \frac {T_k \Gamma^{\alpha_1} u } {\langle r-t\rangle}
 \|_{\infty} \cdot
 \| \langle r-t\rangle \partial^2 \Gamma^{\alpha_2}u \|_2 
 \lesssim  t^{-\frac 32}.
 \end{align}
 If $|\alpha_1|>|\alpha_2|$, then we have
 \begin{align}
 \| T_k \Gamma^{\alpha_1} u  \partial^2 \Gamma^{\alpha_2} u\|_2
 \lesssim
 \| \frac {T_k \Gamma^{\alpha_1} u } {\langle r-t\rangle}
 \|_{2} \cdot
 \| \langle r-t\rangle \partial^2 \Gamma^{\alpha_2}u \|_{\infty}
 \lesssim t^{-1} \cdot  t^{-\frac 12} 
 \lesssim t^{-\frac 32} .
 \end{align}

\texttt{Estimate of $\| \partial \Gamma^{\alpha_1} u T \partial \Gamma^{\alpha_2} u\|_2$}.
If $|\alpha_1| \le |\alpha_2|$ we have
\begin{align}
\| \partial \Gamma^{\alpha_1} u T \partial \Gamma^{\alpha_2} u\|_2
\lesssim \| \partial \Gamma^{\alpha_1} u \|_{\infty}
\cdot \| T \partial \Gamma^{\alpha_2} u \|_2 \lesssim t^{-\frac 32}.
\end{align}
If $|\alpha_1| > |\alpha_2|$ we have
\begin{align}
\| \partial \Gamma^{\alpha_1} u T \partial \Gamma^{\alpha_2} u\|_2
\lesssim \| \partial \Gamma^{\alpha_1} u \|_{2}
\cdot \| T \partial \Gamma^{\alpha_2} u \|_{\infty} \lesssim  t^{-\frac 32}.
\end{align}

Collecting the estimates, we have proved
\begin{align} \label{4.13U}
  &\| \sum_{\substack{\alpha_1<\alpha,  \alpha_2 <\alpha \\ \alpha_1+\alpha_2\le \alpha}}
g^{kij}_{\alpha_1,\alpha_2} \partial_k \Gamma^{\alpha_1} u 
\partial_{ij} \Gamma^{\alpha_2} u    \|_2 \lesssim  t^{-\frac 32}.
\end{align}

\subsection{The case $\al_{2}=\al$.}
Noting that $g^{kij}_{0,\alpha}=g^{kij}$, we have
\begin{align}
 \int g^{kij}\pa_{k}u\pa_{ij}v\pa_{t}v e^{p}
 &= \OK \underbrace{- \int g^{kij}\pa_{jk}u\pa_{i}v\pa_{t}v e^{p}}_{I_1}\underbrace{-\int g^{kij}\pa_{k}u\pa_{i}v\pa_{t}v \pa_{j}(e^{p})}_{I_2}-\int g^{kij}\pa_{k}u\pa_{i}v\pa_{tj}v e^{p}.
\end{align}
Here in the above, the term ``OK" is zero if $\partial_j =\partial_1$ or $\partial_2$. 
This term is nonzero when $\partial_j = \partial_t$, i.e. we should absorb it into the energy when
integrating by parts in the time variable. 

Further integration by parts gives
\begin{align}
  -\int g^{kij}\pa_{k}u\pa_{i}v\pa_{tj}v e^{p}
  &=\OK+ \underbrace{\int g^{kij}\pa_{tk}u\pa_{i}v\pa_{j}v e^{p}}_{I_3}+\underbrace{\int g^{kij}\pa_{k}u\pa_{i}v\pa_{j}v \pa_{t}(e^{p})}_{I_4}+\int g^{kij}\pa_{k}u\pa_{it}v\pa_{j}
  v e^{p}.
  \end{align}
\begin{align}
  \int g^{kij}\pa_{k}u\pa_{it}v\pa_{j}v e^{p}&=\OK \underbrace{-\int g^{kij}\pa_{ik}u\pa_{t}v\pa_{j}v e^{p}}_{I_5}\underbrace{-\int g^{kij}\pa_{k}u\pa_{t}v\pa_{j}v \pa_{i}(e^{p})}_{I_6}-\int g^{kij}\pa_{k}u\pa_{t}v\pa_{ij}v e^{p}.
\end{align}
It follows that
$$2\int g^{kij}\pa_{k}u\pa_{ij}v\pa_{t}v e^{p}=(I_1+I_3+I_5)+(I_2+I_4+I_6) +\OK.$$

Observe that if $\varphi=\pa_{k}u$ or $\varphi=e^{p}$, then 
\begin{align}\label{varphi1}
   &-\pa_{j}\varphi\pa_{i}v\pa_{t}v+\pa_{t}\varphi\pa_{i}v\pa_{j}v-\pa_{i}\varphi\pa_{t}v\pa_{j}v
   \notag\\
  =&-T_{j}\varphi\pa_{i}v\pa_{t}v+\omega_{j}\pa_{t}\varphi\pa_{i}v \pa_{t} v+\pa_{t}\varphi\pa_{i}v\pa_{j}v-T_{i}\varphi\pa_{t}v\pa_{j}v+\omega_{i}\pa_{t}\varphi\pa_{t}v T_{j}v-\omega_{i}\omega_{j}\pa_{t}\varphi(\pa_{t}v)^2 \notag \\
  =&-T_{j}\varphi\pa_{i}v\pa_{t}v+\pa_{t}\varphi\pa_{i}v T_{j} v-T_{i}\varphi\pa_{t}v\pa_{j}v+\omega_{i}\pa_{t}\varphi\pa_{t}v T_{j}v-\omega_{i}\omega_{j}\pa_{t}\varphi(\pa_{t}v)^2 \notag\\
  =&-T_{j}\varphi\pa_{i}v\pa_{t}v+\pa_{t}\varphi T_{i}v T_{j} v-T_{i}\varphi\pa_{t}v\pa_{j}v-\omega_{i}\omega_{j}\pa_{t}\varphi(\pa_{t}v)^2.
 \end{align}
By \eqref{varphi1} and rewriting $\partial_t \varphi = \partial_k \partial_t u
=T_k \partial_t u -\omega_k \partial_{tt} u$, we have
\begin{align}
I_1+I_3+I_5=&\;\int g^{kij}(
-T_{j}\partial_k u \pa_{i}v\pa_{t}v+\pa_{t}\partial_k u T_{i}v T_{j} v-T_{i}\partial_k u\pa_{t}v\pa_{j}v-\omega_{i}\omega_{j}T_k \partial_t u(\pa_{t}v)^2) e^pdx.
\end{align}
By Lemma \ref{lem2.6}, we have $\|T \partial u\|_{\infty} \lesssim t^{-\frac 32} $
and $\| \langle r-t\rangle \partial^2 u \|_{\infty} \lesssim t^{-\frac 12} $. Clearly then
\begin{align}
\int_{\text{$r<\frac t2$ or $r>2t$} } |\partial^2 u| |T v|^2 dx\lesssim t^{-\frac 32} ,
\quad \int_{r \sim t } |\partial^2 u | |Tv|^2 dx 
\ll \int e^p q^{\prime} |Tv|^2 dx. 
\end{align}
It follows that 
\begin{align}
I_1+I_3+I_5 =\OK.
\end{align}
Plugging $\varphi=e^{p}$ in \eqref{varphi1} and noting that $T_j (e^p)=0$, we  have
\begin{align*}
    I_2+I_4+I_6    
    &=\int g^{kij}\pa_k u\Big(-T_j(e^p)\pa_iv \pa_tv- T_i(e^p)\pa_tv\pa_jv-\omega_i\omega_j(\pa_tv)^2\pa_t(e^p)+\pa_t(e^p)T_iv T_jv \Big) \notag \\
  &=   \int g^{kij}\left(  -T_k u\cdot \omega_i\omega_j (\pa_tv)^2 \pa_t(e^p) +\partial_k u\pa_t(e^p)T_i v T_j v \right).
\end{align*}
By Lemma \ref{lem2.6} we have $ \Bigl| |Tu| |\partial_t (e^p)| \Bigr|\lesssim t^{-\frac 32}$. 
Clearly
\begin{align}
\|\partial u \partial_t (e^p)\|_{L_x^{\infty}(r<\frac t2,\, \text{or } r>2t)} \lesssim t^{-\frac 32},
\quad \int_{r\sim t} |\partial u \partial_t (e^p)| |Tv|^2 dx \ll \int e^p q^{\prime} |Tv|^2 dx. 
\end{align}
Thus
\begin{align*}
  I_{2}+I_{4}+I_{6} = \OK. 
\end{align*}
This concludes the case $\alpha_2=\alpha$. In the next section we deal with the main
piece $\alpha_1=\alpha$. 

\section{Estimate of the main piece $\al_{1}=\al$, $\al_{2}=0$}
In this section we estimate the main piece $\alpha_1=\alpha$. By \eqref{2.10A}, we have 
\begin{align*}
 \int g^{kij}\pa_{k}v\pa_{ij}u\pa_{t}v e^{p}
  =&\int g^{kij} (T_{k}v\pa_{ij}u-\omega_{k}\pa_{t}vT_{i}\pa_{j}u+\omega_{k}\omega_{i}\pa_{t}v T_{j}\pa_{t}u)\pa_{t}v e^{p}\\
  =&\int g^{kij} (T_{k}vT_{i}\pa_{j}u-\omega_{i}T_{k}vT_{j}\pa_{t}u+\omega_{i}\omega_{j}T_{k}v\pa_{tt}u-\omega_{k}\pa_{t}vT_{i}\pa_{j}u+\omega_{k}\omega_{i}\pa_{t}v T_{j}\pa_{t}u)\pa_{t}v e^{p}.
\end{align*}
By Lemma \ref{lem2.6}, all terms containing $T\partial u$ decay as $O(t^{-\frac 32} )$.
Thus 
\begin{equation}\label{eq:al1-00}
  \int g^{kij} \pa_{k}v\pa_{ij}u\pa_{t}v e^{p}= \OK+ \int g^{kij} \omega_{i}\omega_{j}T_{k}v\pa_{tt}u\pa_{t}v e^{p}. 
\end{equation}
Recall $T_{0}=0$, $T_{1}=\omega_{1}\pa_{+}-\frac{\omega_{2}}{r}\pa_{\theta}$,  $T_{2}=\omega_{2}\pa_{+}+\frac{\omega_{1}}{r}\pa_{\theta}$. We have
\begin{align*}
 g^{kij} \omega_{i}\omega_{j}T_{k}v&= g^{1ij} \omega_{i}\omega_{j}(\omega_{1}\pa_{+}v-\frac{\omega_{2}}{r}\pa_{\theta}v)
  +g^{2ij} \omega_{i}\omega_{j}(\omega_{2}\pa_{+}v+\frac{\omega_{1}}{r}\pa_{\theta}v)\\
  &= (g^{1ij} \omega_{1}\omega_{i}\omega_{j}+g^{2ij} \omega_{2}\omega_{i}\omega_{j})\partial_+v +\omega_i\omega_j (g^{2ij}\omega_1
  -g^{1ij} \omega_2) \frac 1 r \partial_{\theta} v \notag \\
  &=: h_1(\theta) \partial_+ v + h_2(\theta) \frac 1 r \partial_{\theta} v.\\
\end{align*}
We first estimate the piece 
\begin{align}
\int h_1 (\theta) \partial_+ v \pa_{tt}u\pa_{t}v e^{p}. 
\end{align}
The other piece will be estimated in the next section.

Choose nonnegative radial  $\tilde \phi_1 \in C_c^{\infty}(\mathbb R^2)$ such that $\tilde \phi_1(z)=1$ for 
$\frac 23 \le |z| \le \frac 32$ and $\tilde \phi_1(z)=0$ for $|z|\le \frac 13$ or $|z|\ge 2$.
Denote $\phi(x) =\tilde \phi_1(\frac x t)$. Then
\begin{equation}\label{eq:al1-01}
\begin{split}
  \int h_1(\theta) \partial_+v\pa_{tt}u\pa_{t}v e^{p}
  &= \int h_1 (\theta)\pa_{+}v\pa_{tt}u\pa_{t}v e^{p}\cdot\left(1-\phi\right)+\int h_1(\theta)\pa_{+}v\pa_{tt}u\pa_{t}v e^{p}\phi.
  \end{split}
\end{equation}
By Lemma \ref{lem2.6}, we have
\begin{equation}\label{eq:al1-02}
  \int h(\theta)\pa_{+}v\pa_{tt}u\pa_{t}v e^{p}\cdot\left(1-\phi \right)\lesssim t^{-1}\int |\pa v|^2|\Lg r-t\Rg\pa_{tt}u| \lesssim t^{-\frac 32} = \OK.
\end{equation}
By using the identity $\pa_{t}=\frac{\pa_{+}+\pa_{-}}{2}$ and 
the fact that $\| \langle r -t \rangle \partial^2 u \|_{\infty} \lesssim t^{-\frac 12} $, we get
\begin{align}\label{eq:al1-03}
  2\int h_1(\theta)\pa_{+}v\pa_{tt}u\pa_{t}v e^{p}\phi &=\int h_1(\theta)\pa_{+}v\pa_{tt}u\pa_{+}v e^{p}\phi+\int h_1 (\theta)\pa_{+}v\pa_{tt}u\pa_{-}v e^{p}\phi \notag \\
  &=\OK+ \int h_1(\theta)\pa_{+}v\pa_{tt}u\pa_{-}v e^{p}\phi.
\end{align}
Integrating by parts, we have
\begin{align}
  \int h_1(\theta)\pa_{+}v\pa_{tt}u\pa_{-}v e^{p}\phi  \cdot  rdr d\theta
  &=\;\frac{d}{dt}\int h_1 (\theta)v\pa_{tt}u\pa_{-}v e^{p}\phi dx
   -\int h_1 (\theta)v\pa_{-}v \pa_{+}\left(\pa_{tt}ue^{p}\phi\right) dx
  \notag \\
  & \quad -\int h_1 (\theta)v\pa_{tt}u\pa_{+}\pa_{-}v e^{p}\phi dx
  -\int h_1 (\theta) v \partial_{tt} u \partial_- v e^p \phi \frac 1r dx.
 \end{align}
 In the above computation, one should note that when integrating by parts in $r$ we should
 take into consideration the factor $r$ in the metric $rdr$. The fourth term exactly corresponds to the derivative
 of the metric factor.
The first  and fourth terms are clearly acceptable by using Hardy and the decay of $\langle r-t\rangle \partial_{tt}u$. 
For the second term we have
\begin{align}
  \left| \langle r -t\rangle \pa_{+}\left(\pa_{tt}ue^{p}\phi\right)\right|
  &\lesssim \left|\langle r -t\rangle \pa_{+}\pa_{tt}u\phi\right|+\left|\langle r-t\rangle \pa_{tt}u\pa_{+}\phi\right|  \notag \\
  &\lesssim  t^{-1}\|2\langle r -t \rangle L_0\pa_{tt}u- \langle r-t\rangle ( t-r)\pa_{-}\pa_{tt}u
  \|_{L_x^{\infty}(|x|>\frac t{10})} + t^{-\frac 32}\lesssim t^{-\frac32}. \label{5.6A}
\end{align}
Here in the derivation of \eqref{5.6A}, we used Lemma \ref{lem2.6} and the inequalities
\begin{align}
&|\langle r -t\rangle L_0  \partial_{tt} u|
\lesssim | \langle r -t \rangle \partial_{tt} \Gamma^{\le 1} u |
\lesssim t^{-\frac 12}, \qquad \text{for }r \ge t/10.
\end{align}
For the third term we use the identity $\pa_{+}\pa_{-}v=\Box v+\frac{\pa_{r}v}{r}+\frac{\pa_{\theta\theta}v}{r^2}$ and compute it as
\begin{equation}\label{eq:al1-2}
 \begin{split}
  &\int h_1 (\theta)v\pa_{tt}u\pa_{+}\pa_{-}v e^{p}\phi\\
  =&\int h_1 (\theta)v\pa_{tt}u\left(\frac{\pa_{r}v}{r} +\frac{\pa_{\theta\theta}v}{r^2}\right)e^{p}\phi
  +\sum_{\be_{1}+\be_{2}\leq \al}\int h(\theta)v\pa_{tt}u\cdot g_{\be_{1},\be_{2}}^{kij}\pa_{k}\Ga^{\be_{1}}u\pa_{ij}\Ga^{\be_{2}}ue^{p}\phi.
 \end{split}
\end{equation}
Integrating by parts (for the term $\partial_{\theta\theta}v$), we have
\begin{align*}
  &\int h_1 (\theta)v\pa_{tt}u\left(\frac{\pa_{r}v}{r}+\frac{\pa_{\theta\theta}v}{r^2}\right)e^{p}\phi\\
  =& \int h_1 (\theta) \frac {v}{\langle r -t\rangle}
  \langle r -t\rangle \partial_{tt} u \partial_r v \cdot \frac 1r e^p \phi 
  -\int h_1 (\theta)\pa_{tt}u\left(\frac{\pa_{\theta}v}{r}\right)^2e^{p}\phi-\int \pa_{\theta}(h_1 (\theta)\pa_{tt}u)v\frac{\pa_{\theta}v}{r^2}e^{p}\phi\\
  =& \OK.
  \end{align*}
By \eqref{4.13U}, we have
\begin{equation*}
  \sum_{\be_{1}< \al,\be_{2}< \al,\atop \be_{1}+\be_{2}\leq \al}\int h_1 (\theta)v\pa_{tt}u \cdot g_{\be_{1},\be_{2}}^{kij}\pa_{k}\Ga^{\be_{1}}u\pa_{ij}\Ga^{\be_{2}}u e^{p}\phi\lesssim   t^{-2}=\OK.
\end{equation*}
For the term  $\be_{1}=\al$, $\be_{2}=0$ in \eqref{eq:al1-2}, it follows from \eqref{a2.12a} that
\begin{align*}
  \int g ^{kij}h_1 (\theta)v\pa_{tt}u \pa_{k}v\pa_{ij}u e^{p}\phi
  \lesssim &\int|v\pa_{tt}u||T v\pa^2u|e^{p}\phi+\int|v\pa_{tt}u||\pa v||T\pa u| e^{p}\phi\\
  \lesssim &\int|T v|^2|\pa^2u|e^{p}\phi+t^{-\frac32}\left\|\Lg r-t\Rg^{-1}v\right\|_{L_{x}^2(\R^2)}^2+t^{-2}\\
=& \OK.
\end{align*}

For the term  $\be_{1}=0$, $\be_{2}=\al$ in \eqref{eq:al1-2}, we apply \eqref{2.10A} to obtain
\begin{align*}
  \int g ^{kij}h_1 (\theta)v\pa_{tt}u\pa_{k}u \pa_{ij}v e^{p}\phi
  = \int g ^{kij}h_1 (\theta)v\pa_{tt}u\cdot(T_{k}u\pa_{ij}v-\omega_{k}\pa_{t}uT_{i}\pa_{j}v+\omega_{k}\omega_{i}\pa_{t}u T_{j}\pa_{t}v) e^{p}\phi.
\end{align*}
We rewrite it as
\begin{align*}
  \int g ^{kij}h_1 (\theta)v\pa_{tt}uT_{k}u\pa_{ij}v e^{p}\phi
  =&\int g ^{kij}\pa_{i}(h_1 (\theta)v\pa_{tt}uT_{k}u\pa_{j}v e^{p}\phi)
  -\int g ^{kij}\pa_{i}(h_1 (\theta) T_{k}u e^{p}\phi )v \pa_{tt}u\pa_{j}v\\
  &-\int g ^{kij}h_1(\theta)v\pa_{i}\pa_{tt}uT_{k}u\pa_{j}v e^{p}\phi-\int g ^{kij}h_1 (\theta)\pa_{i}v\pa_{tt}uT_{k}u\pa_{j}v e^{p}\phi.
\end{align*}
The term $\int g ^{kij}\pa_{i}(h_1 (\theta)v\pa_{tt}uT_{k}u\pa_{j}v e^{p}\phi)$ is 
zero for $i\ne 0$. For $i=0$ it is clearly acceptable since it can be absorbed into the time
derivative of the energy due to its smallness. By Lemma \ref{lem2.3a} and \ref{lem2.6}, we have
\begin{align*}
  |\pa_{i}(h_1 (\theta) T_{k}u e^{p}\phi )|&\lesssim |\pa_{i}h_1 (\theta) T_{k}u e^{p}\phi |+|h_1(\theta)\pa_{i} T_{k}u e^{p}\phi |+|h_1 (\theta) T_{k}u \pa_{i}e^{p}\phi |+|h_1 (\theta) T_{k}u e^{p}\pa_{i}\phi |\\
  &\lesssim  t^{-\frac32}+|h_1 (\theta)\pa_{i} \omega_{k}\pa_{t}u e^{p}\phi |+|h_1 (\theta)T_{k}\pa_{i}u e^{p}\phi |+\left|h_1 (\theta) \frac{T_{k}u}{\Lg r-t\Rg} \phi \right|
  \lesssim  t^{-\frac32}.
 \end{align*}
The term containing $v\partial_i \partial_{tt} u$ can be handled by \eqref{2.30D}. Thus
\begin{align*}
  &\int g ^{kij}h_1 (\theta)v\pa_{tt}uT_{k}u\pa_{ij}v e^{p}\phi
  = \OK.
\end{align*}
Similarly, we have
\begin{align*}
  &\int g ^{kij}\omega_{k}h_1(\theta)v\pa_{tt}u\pa_{t}uT_{i}\pa_{j}v e^{p}\phi
  =\int g ^{kij}\omega_{k}h_1 (\theta)v\pa_{tt}u\pa_{t}u\left(\pa_{j}T_{i}v-\pa_{j}\omega_{i}\pa_{t}v \right) e^{p}\phi =\OK,\\
  &\int g ^{kij}\omega_{k}\omega_{i}h_1 (\theta)v\pa_{tt}u\pa_{t}uT_{j}\pa_{t}v e^{p}\phi=\int g ^{kij}\omega_{k}\omega_{i}h_1 (\theta)v\pa_{tt}u\pa_{t}u\pa_{t}T_{j}v e^{p}\phi = \OK.
\end{align*}
This concludes the estimate of the first part of the main piece.

\section{further estimates}
We now denote $h(\theta)=h_2(\theta)$ and consider the second part of the main piece 
\begin{align}
\int h(\theta) \frac 1 r \partial_{\theta} v \partial_{tt} u \partial_t v e^p dx.
\end{align}
Since $\| \langle r -t\rangle \partial_{tt} u \|_{\infty}
\lesssim t^{-\frac 12}$, it follows that
\begin{align}
&\int h(\theta) \frac 1 r \partial_{\theta} v \partial_{tt} u \partial_t v e^p dx 
= \OK+\int h(\theta) \frac 1 r \partial_{\theta} v \partial_{tt} u \partial_t v \tilde \phi (\frac x t) e^p dx,
\end{align}
where $\tilde \phi $ is a radial bump function  localized to $|z|\sim 1$. Denote 
$\phi(z)=\frac 1{|z|} \tilde \phi(z)$. Then 
\begin{align}
& \int h(\theta) \frac 1 r \partial_{\theta} v \partial_{tt} u \partial_t v \tilde \phi (\frac x t) e^p dx \notag \\
=& \frac 1t \int h(\theta) \partial_{\theta} v \partial_{tt} u \partial_t v \phi (\frac x t) e^p dx \notag \\
=& \OK+ \frac 1 t \int h(\theta) v \partial_{tt} u \partial_t \partial_{\theta} v \phi
(\frac x t) e^p dx \notag \\
=& \OK + \frac 1 t \int   \underbrace{ h(\theta) \langle r -t \rangle  \partial_{tt} u \phi(\frac x t) 
e^p }_{=: F} \underbrace{ \frac {v} {\langle r -t \rangle} }_{=: \tilde v} \partial_t \partial_{\theta} v dx. 
\end{align}
 Note that 
\begin{align}
&\| \partial F \|_{\infty} + \|F \|_{\infty} \lesssim t^{-\frac 12} E_4^{\frac 12},
\quad \| \tilde v \|_2 + \| \nabla \tilde v \|_2 \lesssim E_m^{\frac 12}; \\
& \| \langle \nabla 
\rangle ( F \tilde v) \|_2 \lesssim \| F \tilde v \|_2 + \| \nabla (F \tilde v ) \|_2 \lesssim t^{-\frac 12}
E_4^{\frac 12} E_m^{\frac 12}.
\end{align}
It follows that
\begin{align}
\Bigl| \int F \tilde v  \partial_t \partial_{\theta} v dx\Bigr|
\lesssim t^{-\frac 12} E_4^{\frac 12} E_m^{\frac 12} \| \langle \nabla \rangle^{-1} \partial_t \partial_{\theta} v \|_2.
\end{align}
Recall that $v=\Gamma^{\alpha} u$ with $|\alpha| \le m$. 
Thus we only need to show (below we take $0< \delta<1/4$)
\begin{align}
\| \langle \nabla \rangle^{-1} \partial_t \Gamma^{\le m+1} u \|_2 \le D_1 t^{\delta},
\quad 
\end{align}
where $D_1$ is a small constant whose smallness can be ensured by the smallness of $E_m$. 
The legitimacy of the nonlocal norm $\| \langle \nabla \rangle^{-1} \partial
\Gamma^{\le m+1} u \|_2$ is ensured by Lemma \ref{lem_Nonlocal1}.

\subsection{Estimate of $\| \langle \nabla \rangle^{-1} \partial  \Gamma^{\le m+1} u \|_2$}
For each multi-index $\beta$ with $|\beta|\le m+1$, we have
\begin{align}
\square \Gamma^{\beta} u
= \sum_{\alpha_1+\alpha_2 \le \beta}
g^{kij}_{\beta;\alpha_1,\alpha_2} \partial_k \Gamma^{\alpha_1} u 
\partial_{ij} \Gamma^{\alpha_2} u,
\end{align}
where $g^{kij}_{\beta,\alpha_1,\alpha_2}$ still satisfies the null conditions 
for each ($\beta$, $\alpha_1$, $\alpha_2$). Moreover
$g^{kij}_{\beta,\beta,0} =g^{kij}_{\beta,0,\beta} = g^{kij}$. 

We first compute the left hand side. By using the Littlewood-Paley decomposition (see
\eqref{LP_def1}), we have
\begin{align}
&\sum_{J\ge 0} \sum_{|\beta| \le  m+1} 2^{-2J} \int \square P_J \Gamma^{\beta} u \partial_t P_J \Gamma^{\beta} u e^p dx  \notag \\
=&\sum_{J\ge 0} \sum_{|\beta|\le m+1} 2^{-2J}\Bigl( \frac 12  \frac d {dt}
\| e^{\frac p2} \partial P_J \Gamma^{\beta} u \|_2^2 +
\frac 12 \int e^p q^{\prime} |T P_J \Gamma^{\beta} u|^2 dx \Bigr). \label{e5.8S0}
\end{align}
It is not difficult to check that 
\begin{align}
\sum_{J\ge 0}
\sum_{|\beta|\le m+1}
2^{-2J}
\| e^{\frac p2} \partial P_J \Gamma^{\beta} u \|_2^2 
\sim \sum_{|\beta|\le m+1} \| \langle \nabla \rangle^{-1} \partial \Gamma^{\beta} u\|_2^2.
\end{align}

To simplify the notation in the subsequent nonlinear estimates, we introduce the following
terminology.

\noindent
\textbf{Notation}.  For a quantity $X(t)$, we shall write $X(t)= \mathrm{NICE}$  if $X(t)$ can be written as
\begin{align} \label{NiceNot}
X(t)= \frac d {dt} X_1 (t)+X_2(t) +X_3(t), 
\end{align}
where  (below $\alpha_0>0$ is some constant)
\begin{align}
&|X_1(t)| \ll 
\sum_{|\beta|\le m+1}
\| (\langle \nabla \rangle^{-1} \partial \Gamma^{\beta}  u)(t,\cdot) \|_{L_x^2(\mathbb R^2)}^2, 
\quad |X_2(t)| \ll    \sum_{J\ge 0} \sum_{|\beta| \le  m+1} 
2^{-2J}\int e^p q^{\prime} |(T P_J \Gamma^{\beta}  u)(t,x) |^2 dx ;\notag \\
&|X_3(t)| \lesssim \langle t \rangle^{-1-\alpha_0}.
\end{align}

\medskip

\medskip

Next we shall deal with the RHS, namely 
\begin{align}
\sum_{|\beta| \le m+1} 
\sum_{\alpha_1+\alpha_2\le \beta}
\sum_{J\ge 0} 2^{-2J}
\Bigl( g^{kij}_{\beta;\alpha_1,\alpha_2} 
\int P_J( \partial_k \Gamma^{\alpha_1} u \partial_{ij} \Gamma^{\alpha_2} u )
\partial_t P_J (\Gamma^{\beta} u ) e^p dx\Bigr).
\end{align}

We shall discuss several cases. To simplify the notation, we fix $\beta$ and denote $
w=\Gamma^{\beta} u$.   The most difficult case is the quasilinear piece which will be discussed
in detail below.

 Case 1: \underline{the quasilinear piece $\alpha_1=0$, $\alpha_2=\beta$}. 
In this case we need to estimate
 \begin{align}
\sum_{J\ge 0} 2^{-2J} g^{kij} \int P_J( \partial_k u \partial_{ij} w) \partial_t P_J w e^p dx.
\end{align}
 We discuss several further subcases.

 Case 1a: the piece 
 \begin{align}
 &\sum_{J\ge 8} 2^{-2J} g^{kij} \int P_J ( \partial_k u \partial_{ij} P_{[J-3, J+3]} w)
 \partial_t P_J w e^p dx \notag \\
= &\sum_{J\ge 8} 2^{-2J} g^{kij} \int \partial_k u \partial_{ij} P_J w
 \partial_t P_J w e^p dx \label{e5.34S1}\\
& \quad +\sum_{J\ge 8} 2^{-2J} g^{kij} \int \Bigl([P_J,  \partial_k u ]\partial_{ij} P_{[J-3,J+3]} w\Bigr)
 \partial_t P_J w e^p dx. \label{e5.34S0}
 \end{align}

 It is not difficult to check that the contribution of \eqref{e5.34S1} is acceptable for us. We now
 focus on the estimate of \eqref{e5.34S0}.  For simplicity of notation, we denote
 \begin{align}
\boxed{w_J =P_{[J-3, J+3]}w.}
\end{align}

 Clearly
 \begin{align}
  & \sum_{J\ge 8} 2^{-2J} g^{kij}  \int \Bigl([P_J,  \partial_k u ]\partial_{ij} w_J\Bigr)
 \partial_t P_J w e^p dx \notag \\
 = & \sum_{J\ge 8} 2^{-2J} g^{kij} \int \int 2^{2J} \varphi(2^J y) ( (\partial_k u)(x-y) - (\partial_k u)(x) )
 (\partial_{ij} w_J)(x-y) dy \partial_t P_J w e^p dx \notag \\
 =& \sum_{J\ge 8} 2^{-2J} g^{kij}\sum_{m=1}^2 \int \int \int_0^1  2^J \varphi_m(2^J y)
 (\partial_m \partial_k  u)(x-\theta y) 
 (\partial_{ij} w_J)(x-y) \partial_t P_J w e^p d\theta dy dx, \label{e5.38S0}
 \end{align} 
 where $\varphi$ and $\phi_m$ are Schwartz functions. Here $2^{2J} \varphi(2^J \cdot)$ is the kernel function corresponding to $P_J$. For $J=0$ and $J\ge 1$ we have slightly different expressions for $\varphi$.
But we shall ignore this difference for simplicity of notation. 
 
 We first need an auxiliary estimate. 
 \begin{lem} \label{Lem5.2S0}
 We have
 \begin{align}
&\sum_{i=1}^2 \| \partial \partial_i P_{\le J+3} w \|_2 \lesssim 2^J \| \partial P_{\le J+3} w \|_2
; \notag \\
& \| \square P_{\le J+3} w \|_2 \lesssim 
2^J t^{-\frac 12} \| \partial P_{\le J+5} w \|_2
+t^{-\frac 12} 
\| \partial P_{\le J+5} \Gamma^{\le m+1} u \|_2
+ t^{-\frac 12} 
\| \langle \nabla \rangle^{-1} \partial \Gamma^{\le m+1} u\|_2;  \notag\\
& \| \partial_{tt} P_{\le J+3} w \|_2 \lesssim 2^J \| \partial P_{\le J+3} w \|_2
 +2^J t^{-\frac 12} \| \partial P_{\le J+5} w \|_2
\notag \\
& \qquad \qquad \qquad +t^{-\frac 12}
\| \partial P_{\le J+5} \Gamma^{\le m+1} u \|_2
+ t^{-\frac 12} 
\| \langle \nabla \rangle^{-1} \partial \Gamma^{\le m+1} u\|_2. \notag
 \end{align}
 The same estimates hold when $P_{\le J+3} w$ on the LHS above  is replaced
 by $\boxed{w_J=P_{[J-3,J+3]} w}$. 
 \end{lem}
 \begin{proof}
 The first estimate is obvious. We only need to show the second estimate
 since the third estimate follows from the identity $\partial_{tt} =\square +\Delta$.
 Observe that (for simplicity denote $g^{kij}_{\alpha_1,\alpha_2}=g^{kij}_{\beta;\alpha_1,\alpha_2}$)
\begin{align}
\square w
= \sum_{\alpha_1+\alpha_2\le \beta} g^{kij}_{\alpha_1,\alpha_2}
 \partial_k \Gamma^{\alpha_1} u \partial_{ij} \Gamma^{\alpha_2} u.
\end{align}
The main difficult term on the RHS is the case $\partial_{ij} =\partial_{tt}$,
$\alpha_2=\beta$. We rewrite the above as
\begin{align}
\square w = g^{k00} \partial_k u ( \square w +\Delta w) +
\sum_{\alpha_1+\alpha_2\le \beta, \alpha_2<\beta}
g^{kij}_{\alpha_1,\alpha_2} \partial_k \Gamma^{\alpha_1} u
\partial_{ij} \Gamma^{\alpha_2} u+
\sum_{(i,j)\ne (0,0)} g^{kij} \partial_k u \partial_{ij} w.
\end{align}
Thus (below the Einstein summation convention is still in force, e.g.
$g^{k00} \partial_k u = \sum_{k=0}^2 g^{k00} \partial_k u $)
\begin{align}
\square w =  \frac 1 { 1- g^{k00} \partial_k u}  ( g^{k00} \partial_k u \Delta w+
\sum_{\alpha_1+\alpha_2 \le \beta, \alpha_2<\beta}  g^{kij}_{\alpha_1,\alpha_2}
\partial_k \Gamma^{\alpha_1} u \partial_{ij} \Gamma^{\alpha_2} u
+ \sum_{(i,j) \ne (0,0)} g^{kij} \partial_k u \partial_{ij} w).
\end{align}
Denote $\tilde f = \frac {g^{k00} \partial_k u} { 1-g^{k00} \partial_k u}$.
Since $\| \partial \Gamma^{\le 3} u \|_{\infty} \lesssim t^{-\frac 12} E_5^{\frac 12}$, we have
$\| \partial^{\le 3} \tilde f \|_{\infty} \lesssim t^{-\frac 12}$. 
Clearly \begin{align}
\| P_{\le J+3} ( \tilde f \Delta w)
\|_2 &
\le \| P_{\le J+3} ( \tilde f P_{\le J+5}\Delta w) \|_2
+ \|  P_{\le J+3} ( \tilde f P_{\ge J+6} \Delta w) \|_2 \notag \\
& \lesssim\; 2^J t^{-\frac 12} \| \partial P_{\le J+5} w \|_2
+ t^{-\frac 12} \| \langle \nabla \rangle^{-1} \partial w \|_2.
\end{align}
By a similar estimate, we have
\begin{align}
\| P_{\le J+3}
( \frac 1 {1- g^{k00} \partial_k u} \sum_{(i,j)\ne (0,0)}
g^{kij} \partial_k u \partial_{ij} w ) \|_2
\lesssim \;2^J t^{-\frac 12} \| \partial P_{\le J+5} w \|_2
+ t^{-\frac 12} \| \langle \nabla \rangle^{-1} \partial w \|_2.
\end{align}
To estimate $\| P_{\le J+3}
\Bigl( \frac 1 {1-g^{k00}\partial_k u}
\sum_{\alpha_1+\alpha_2\le \beta,
\alpha_2< \beta}
g^{kij}_{\alpha_1,\alpha_2}
\partial_k \Gamma^{\alpha_1}u \partial_{ij} \Gamma^{\alpha_2} u\Bigr) \|_2$,
we denote $\tilde f_2= \frac 1 {1-g^{k00}\partial_k u}$ and consider the general expression
\begin{align}
\| P_{\le J+3} ( \tilde f_2 \partial \Gamma^{\alpha_1} u
\partial^2 \Gamma^{\alpha_2} u ) \|_2, 
\quad \alpha_1+\alpha_2 \le \beta, \; \alpha_2 <\beta.
\end{align}
We discuss a few cases. Recall $|\beta|\le m+1$,
$\| \partial \Gamma^{\le m-2} u \|_{\infty} \lesssim t^{-\frac 12} E_m^{\frac 12}$,
and $\| \partial^2 \Gamma^{\le m-3} u \|_{\infty} \lesssim t^{-\frac 12} E_m^{\frac 12}$.

Case 1: $|\alpha_2|=m$ or $|\alpha_2|=m-1$.  Clearly $|\alpha_1|\le 2$ and we have
\begin{align}
& \| P_{\le J+3} ( \tilde f_2 \partial \Gamma^{\le 2} u  \partial^2 \Gamma^{\le m} u ) \|_2
\notag \\
\lesssim & \| \tilde f_2 \partial\Gamma^{\le 2} u \|_{\infty}
\| \partial^2 P_{\le J+5} \Gamma^{\le m} u\|_2
+\sum_{l\ge J+6}
\| P_l ( \tilde f_2 \partial \Gamma^{\le 2} u ) \|_{\infty}
\| \tilde P_l( \partial^2  \Gamma^{\le m} u ) \|_2 \notag \\
\lesssim & t^{-\frac 12}
\Bigl( \| \partial P_{\le J+5} \Gamma^{\le m+1} u \|_2
+ \| \partial \langle \nabla \rangle^{-1} \Gamma^{\le m+1} u \|_2 \Bigr).
\end{align}

Case 2: $|\alpha_1| \le |\alpha_2| \le m-2$. We have
\begin{align}
  & \| P_{\le J+3} ( \tilde f_2  \partial \Gamma^{\alpha_1} u \partial^2 \Gamma^{\alpha_2} u
  ) \|_{2} \notag \\
  \lesssim & \; \| \tilde f_2 \partial \Gamma^{\alpha_1 } u \|_{\infty}
  \| P_{\le J+5} ( \partial^2 \Gamma^{\alpha_2} u ) \|_2 
  + \sum_{l\ge J+6}
  \| P_l( \tilde f_2 \partial \Gamma^{\alpha_1} u ) \|_{\infty} \| \tilde P_l ( \partial^2 \Gamma^{\alpha_2}
  u) \|_{2} \notag \\
  \lesssim &\; t^{-\frac 12}
\Bigl( \| \partial P_{\le J+5} \Gamma^{\le m+1} u \|_2
+ \| \partial \langle \nabla \rangle^{-1} \Gamma^{\le m+1} u \|_2 \Bigr).
  \end{align}

Case 3: $|\alpha_2| <|\alpha_1| \le m-2$. We have
\begin{align}
  & \| P_{\le J+3} ( \tilde f_2  \partial \Gamma^{\alpha_1} u \partial^2 \Gamma^{\alpha_2} u
  ) \|_{2} \notag \\
  \lesssim & \;
  \| P_{\le J+5} ( \partial\Gamma^{\alpha_1} u ) \|_2 
  \| \tilde f_2 \partial^2\Gamma^{\alpha_2 } u \|_{\infty}
  + \sum_{l\ge J+6}
  \| P_l( \tilde f_2 \partial^2 \Gamma^{\alpha_2} u ) \|_{\infty} \| \tilde P_l ( \partial \Gamma^{\alpha_1}
  u) \|_{2} \notag \\
  \lesssim &t^{-\frac 12}
\Bigl( \| \partial P_{\le J+5} \Gamma^{\le m+1} u \|_2
+ \| \partial \langle \nabla \rangle^{-1} \Gamma^{\le m+1} u \|_2 \Bigr). \notag
  \end{align}

Case 4: $|\alpha_1|=m-1$, $|\alpha_2|\le 2$, or $|\alpha_1|=m$, $|\alpha_2|\le 1$,
or $|\alpha_1|=m+1$, $|\alpha_2|=0$.  Easy to check that we also have 
\begin{align}
  & \| P_{\le J+3} ( \tilde f_2  \partial \Gamma^{\alpha_1} u \partial^2 \Gamma^{\alpha_2} u
  ) \|_{2} \notag \\
  \lesssim & \;
  \| P_{\le J+5} ( \partial\Gamma^{\alpha_1} u ) \|_2 
  \| \tilde f_2 \partial^2\Gamma^{\alpha_2 } u \|_{\infty}
  + \sum_{l\ge J+6}
  \| P_l( \tilde f_2 \partial^2 \Gamma^{\alpha_2} u ) \|_{\infty} \| \tilde P_l ( \partial \Gamma^{\alpha_1}
  u) \|_{2} \notag \\
  \lesssim &t^{-\frac 12}
\Bigl( \| \partial P_{\le J+5} \Gamma^{\le m+1} u \|_2
+ \| \partial \langle \nabla \rangle^{-1} \Gamma^{\le m+1} u \|_2 \Bigr). \notag
  \end{align}
The desired estimate then easily follows. 
\end{proof}
 
 We now continue the estimate of \eqref{e5.38S0}. 
 In \eqref{e5.38S0}, it suffices for us to treat the
 case $m=1$ since the estimate for $m=2$ is similar.  We  write
 \begin{align}
 &\sum_{J\ge 8} 2^{-2J} g^{kij}\int \int \int_0^1  2^J \varphi_1(2^J y)
 (\partial_1\partial_k  u)(x-\theta y) 
 (\partial_{ij}  w_J)(x-y) \partial_t P_J w e^p d\theta dy dx \notag \\
 = & \sum_{J\ge 8} 2^{-2J} g^{kij}\int \int \int_0^1  2^J \varphi_1(2^J y) \chi(t^{-\frac 23} y)
 (\partial_1\partial_k  u)(x-\theta y) 
 (\partial_{ij} w_J)(x-y) \partial_t P_J w e^p d\theta dy dx  \label{e5.48S0}\\
 & \quad +\sum_{J\ge 8} 2^{-2J} g^{kij} \int \int \int_0^1  2^J \varphi_1(2^J y)\cdot  (1-\chi(t^{-\frac 23} y) )
 (\partial_1\partial_k  u)(x-\theta y) 
 (\partial_{ij} w_J)(x-y) \partial_t P_J w e^p d\theta dy dx, \label{e5.48S1}
 \end{align} 
 where $\chi \in C_c^{\infty}(\mathbb R^2)$ satisfies $\chi(z)\equiv 1$ for $|z|\le 0.01$
 and $\chi(z)\equiv 0$ for $|z|\ge 0.02$.  In yet other words the cut-off function 
 $\chi(t^{-\frac 23} y)$ is
 to localize $y$ to the regime $|y| \ll t^{\frac 23}$.  In \eqref{e5.48S1},
 since $|y|\gtrsim t^{\frac 23}$, we clearly have (by using Lemma 
 \ref{Lem5.2S0})
 \begin{align}
 | {\eqref{e5.48S1}}| \lesssim \sum_{J\ge 0} 2^{-10J} t^{-10} \| \partial \langle \nabla \rangle^{-1}
 \Gamma^{\le m+1} u \|_2^2.
 \end{align}
 The contribution of this term is clearly acceptable for us.
 
 To estimate \eqref{e5.48S0}, we choose $\phi_1(t,x)=a(x/t)$ where $a\in C_c^{\infty}(\mathbb R^2)$
 is   such that $a(x)=1$ for $ 0.9\le |x|\le 1.1$, and $a(x)=0$ for $|x|\le 0.8$ or $|x|\ge 1.2$.
 We decompose
 \eqref{e5.48S0} as
 \begin{align}
 &\eqref{e5.48S0} \notag \\
  =&\sum_{J\ge 8} 2^{-2J} g^{kij}
 \int \int \int_0^1  2^J \varphi_1(2^J y) \chi(t^{-\frac 23} y) (1-\phi_1(t,x))
 (\partial_1\partial_k  u)(x-\theta y) 
 (\partial_{ij} w_J)(x-y) \partial_t P_J w e^p d\theta dy dx \label{e5.51S0}\\
 & \quad +\sum_{J\ge 8} 2^{-2J} g^{kij}\int \int \int_0^1  2^J \varphi_1(2^J y) \chi(t^{-\frac 23} y) \phi_1(t,x)
 (\partial_k\partial_1  u)(x-\theta y) 
 (\partial_{ij} w_J)(x-y) \partial_t P_J w e^p d\theta dy dx. \label{e5.51S1}
 \end{align}
 Observe that in \eqref{e5.51S0}, 
 since $|y| \ll t$ and $|x|$ is away from the light cone, the variable $x-\theta y$
 is also away from the light cone. We have
 \begin{align}
\sup_{0\le \theta \le 1} \| \chi(t^{-\frac 23} y) (1-\phi_1(t,x))
 (\partial_1\partial_k  u)(x-\theta y)  \|_{L_x^{\infty} L_y^{\infty} } \lesssim t^{-\frac 32} 
 E_5^{\frac 12}.
 \end{align}
 By using this estimate together with Lemma \ref{Lem5.2S0}, it is not difficult to check
 that the contribution of \eqref{e5.51S0} is acceptable for us.  It remains
 for us to estimate \eqref{e5.51S1}. In this case observe that $|x| \sim t$, $|y| \ll t$,
 $|y| \ll |x|$. 
 
 We shall use the identity:
\begin{align}
g^{kij} \partial_k a \partial_{ij} b
=g^{kij} (T_k a \partial_{ij} b - \omega_k \partial_t a T_i \partial_j b + 
\omega_k \omega_i \partial_t a T_j \partial_t b).
\end{align}
One has to be extremely careful here due to the shifts in $x$ induced by convolution! In particular
\begin{align}
T_k ( a(x+h) ) \ne (T_k a) (x+h).
\end{align}

In \eqref{e5.51S1}, we shall apply the above identity with 
\begin{align}
a(x) = (\partial_1 u)(x-\theta y), 
\quad b(x) = w_J(x-y).
\end{align}
 
Subcase 1: the piece 
\begin{align}
\sum_{J\ge 8} 2^{-2J} g^{kij} \int \int \int_0^1 2^J \varphi_1(2^J y) \chi(t^{-\frac 23} y)
\phi_1(t,x)
T_k a \partial_{ij} b \partial_t P_J w e^p d \theta dy dx.
\end{align}
Observe that 
\begin{align}
T_k a = \Bigl( \omega_k (x) \partial_t + \partial_{x_k} \Bigr)
\Bigl(  (\partial_1 u) (x-\theta y) \Bigr). \notag 
\end{align}
Since $|x| \sim t$ and $|y|\ll |x|$,  we have
\begin{align}
| \omega_k (x) - \omega_k (x-\theta y) | \lesssim \frac 1 t \cdot |y|.
\end{align}
Thus we only need to work with the piece
\begin{align} \label{e5.59S0}
\sum_{J\ge 8} 2^{-2J} g^{kij} \int \int \int_0^1 2^J \varphi_1(2^J y) \chi(t^{-\frac 23} y)
\phi_1(t,x)
(T_k \partial_1 u)(x-\theta y) (\partial_{ij}  w_J)(x-y) \partial_t P_J w e^p d \theta dy dx.
\end{align}

Since $\|T \partial  u \|_{\infty} \lesssim t^{-\frac 32}$,  the contribution of the 
term \eqref{e5.59S0} is clearly acceptable for us with the help of Lemma 
\ref{Lem5.2S0}.

 Subcase 2: the piece 
 \begin{align}
 \sum_{J\ge 8} 2^{-2J} \sum_{1\le j\le 2} g^{kij} \int \int \int_0^1 2^J \varphi_1(2^J y) \chi(t^{-\frac 23} y)
\omega_k(x) \phi_1(t,x)
\partial_t a  T_i\partial_j b \partial_t P_J w e^p d \theta dy dx,
\end{align}

Here we only treat the case $j\ne 0$, i.e. we deal with $T_i \nabla b$. Note that 
\begin{align}
&\partial_t a = (\partial_t \partial_1 u)(x-\theta y); \\
 & T_i \partial_j b =
 ( \omega_i (x) \partial_t + \partial_{x_i} ) \Bigl( 
 (\partial_j w_J) (x-y) \Bigr).
 \end{align}
 
Since $|x| \sim t$ and $|y| \ll |x|$,  the contribution of the difference
$\omega_i(x) -\omega_i(x-y)$ is acceptable for us. Thus
we only need to  estimate (for $j=1$ or $j=2$)
\begin{align}
\sum_{J\ge 8} 2^{-2J}  \int \int 
2^J \varphi_1(2^J y) \chi(t^{-\frac 23} y) \phi_1(t,x) \omega_k(x)
(\partial_t  \partial_1 u) (x-\theta y) 
(T_{i} \partial_j w_J)(x-y) \partial_t P_J w e^p  dy dx.
\label{e5.63S0}
\end{align}

Here and below we shall neglect the integral in $\theta$ since the estimates
will be uniform in $\theta \in [0,1]$. 
In \eqref{e5.63S0}, note that $|x-y| \sim t$ and the contribution of the 
commutator (below $z=x-y$)
\begin{align}
([T_i, \partial_j] w_J )(z)= - \Bigl(\partial_{z_j} (\omega_i(z) ) \Bigr) 
\cdot (\partial_t w_J)(z)
\end{align}
is clearly acceptable for us (since $\| \partial_{z_j} (\omega_i(z) ) \|_{L^{\infty}
(|z|\sim t) } \lesssim \frac 1t$). Thus we only need to estimate
\begin{align}
\sum_{J\ge 8} 2^{-2J} \int \int 
2^J \varphi_1(2^J y) \chi(t^{-\frac 23} y) \phi_1(t,x) \omega_k(x)
(\partial_t  \partial_1 u) (x-\theta y) 
(\partial_j T_{i}  w_J)(x-y) \partial_t P_J w e^p  dy dx.
\label{e5.65S0}
\end{align}
We now write $(\partial_j T_{i} w_J)(x-y)
= -\partial_{y_j} \Bigl( (T_i w_J)(x-y ) \Bigr)$. Integrating by
parts in $\partial_{y_j}$, we obtain
\begin{align}
& \eqref{e5.65S0} \notag \\
=&\sum_{J\ge 8} 2^{-2J} \int \int 
2^J \varphi_1(2^J y) \chi(t^{-\frac 23} y) \phi_1(t,x) \omega_k(x)
(\partial_t  \partial_j \partial_1 u) (x-\theta y) \cdot (-\theta)
( T_{i} w_J)(x-y) \partial_t P_J w e^p dy dx \label{e5.66S0} \\
& +\sum_{J\ge 8} 2^{-2J}  \int \int 
2^{2J} (\partial_j \varphi_1)(2^J y) \chi(t^{-\frac 23} y) \phi_1(t,x) \omega_k(x)
(\partial_t   \partial_1 u) (x-\theta y) 
( T_{i} w_J)(x-y) \partial_t P_J w e^p  dy dx \label{e5.66S1}  \\
& +\sum_{J\ge 8} 2^{-2J}  \int \int 
2^J \varphi_1(2^J y)  t^{-\frac 23} (\partial_j \chi)(t^{-\frac 23} y) \phi_1(t,x) 
\omega_k(x)
(\partial_t   \partial_1 u) (x-\theta y) 
( T_{i} w_J)(x-y) \partial_t P_J w e^p  dy dx. \label{e5.66S2} 
\end{align}

For \eqref{e5.66S1}, we have (below $\delta_1>0$ is a small constant)
\begin{align}
 & |\eqref{e5.66S1} | \notag\\
 \le &
 \sum_{J\ge 8} 2^{-2J} 
 \int \int  2^{2J}|(\nabla \varphi_1)(2^J y)| |\phi_1(t,x)| 
 \Bigl( \epsilon {|(T w_J)(x-y)|^2}  \langle |x-y|-t \rangle^{-1-\delta_1}
 \notag \\
 & \qquad
 + C_{\epsilon} \langle |x-y|-t \rangle^{1+\delta_1}  |(\partial^2 u)(x-\theta y)|^2
 |\chi(t^{-\frac 23} y)|^2
 |(\partial_t P_J w)(x)|^2 \Bigr) dx dy , \label{e5.69S0}
\end{align} 
where $\epsilon>0$ can be taken sufficiently small, and $C_{\epsilon}>0$
depends on $\epsilon$.  Note that
\begin{align}
\langle |x-y|-t \rangle^{1+\delta_1}  |(\partial^2 u)(x-\theta y)|^2 
\lesssim \frac {E_5} t ( 1+ |y| ).
 \end{align}
Due to the cut-off function $|(\nabla \varphi_1)(2^J y)|$, the factor $(1+|y|)$ is certainly
harmless for us.  It is then not difficult to check that
the contribution of  \eqref{e5.69S0} is acceptable for us.

It is not difficult to check that the contribution of the term 
\eqref{e5.66S2} is acceptable for us.

The estimate of \eqref{e5.66S0} follows along similar lines. We omit the details.

 Subcase 3: the piece 
 \begin{align}
 \sum_{J\ge 8} 2^{-2J}  \int \int \int_0^1 2^J \varphi_1(2^J y) \chi(t^{-\frac 23} y)
\phi_2(t,x)
\partial_t a  T_j \partial_t b \partial_t P_J w e^p d \theta dy dx,
\end{align}
where $j=1$ or $j=2$, and $\phi_2(t,x)$ is localized to $|x|\sim t$.  Here 
$\phi_2(t,x)$ corresponds to $\phi_1(t,x) \omega_k(x)$ or
$\phi_1(t,x) \omega_k(x) \omega_i(x)$. 
Recall  $b(x) = w_J(x-y)$ and note that 
\begin{align}
(T_j \partial_t b)(x) - (T_j \partial_t  w_J)(x-y) = 
(\omega_j(x) -\omega_j(x-y))  (\partial_{tt}  w_J)(x-y). \label{e5.71S0}
\end{align}
Since $|x|\sim t$ and $|y|\ll t$, the contribution of \eqref{e5.71S0} is acceptable by using Lemma \ref{Lem5.2S0}. Thus we only need to estimate 
 \begin{align}
 \sum_{J\ge 8} 2^{-2J}  \int \int  2^J \varphi_1(2^J y) \chi(t^{-\frac 23} y)
\phi_2(t,x)
(\partial_t \partial_1 u)(x-\theta y)  (T_j \partial_t w_J)(x-y) \partial_t P_J w e^p dy dx.
\label{e5.72S0}
\end{align}
We rewrite \eqref{e5.72S0} as
\begin{align}
 & \eqref{e5.72S0} \notag \\
 =& \frac d {dt} 
 \Bigl(  \sum_{J\ge 8} 2^{-2J}  \int \int  2^J \varphi_1(2^J y) \chi(t^{-\frac 23} y)
\phi_2(t,x)
(\partial_t \partial_1 u)(x-\theta y)  (T_j  w_J)(x-y) \partial_t P_J w e^p dy dx
\Bigr) \label{e5.73S0} \\
& \; - \sum_{J\ge 8} 2^{-2J}  \int \int  2^J \varphi_1(2^J y)  \partial_t ( \chi(t^{-\frac 23} y)
\phi_2(t,x) ) 
(\partial_t \partial_1 u)(x-\theta y)  (T_j  w_J)(x-y) \partial_t P_J w e^p dy dx  \label{5.53d1}\\
& \;-  \sum_{J\ge 8} 2^{-2J}  \int \int  2^J \varphi_1(2^J y)  \chi(t^{-\frac 23} y)
\phi_2(t,x) 
(\partial_{tt} \partial_1 u)(x-\theta y)  (T_j  w_J)(x-y) \partial_t P_J w e^p dy dx  \label{5.53d2}\\
& \; - \sum_{J\ge 8} 2^{-2J}  \int \int  2^J \varphi_1(2^J y)  \chi(t^{-\frac 23} y)
\phi_2(t,x) 
(\partial_t \partial_1 u)(x-\theta y)  (T_j  w_J)(x-y) \partial_{tt} P_J w e^p dy dx   \label{5.53d3}\\
& \; - \sum_{J\ge 8} 2^{-2J}  \int \int  2^J \varphi_1(2^J y)  \chi(t^{-\frac 23} y)
\phi_2(t,x) 
(\partial_t \partial_1 u)(x-\theta y)  (T_j  w_J)(x-y) \partial_{t} P_J w e^p  \partial_t pdy dx. \label{5.53d4}
\end{align}
It is not difficult to check that 
\begin{align}
|\eqref{5.53d1}| + |\eqref{5.53d2} | = \mathrm{NICE}.
\end{align}
For \eqref{5.53d3} we can choose $\tilde \phi_1 \in C_c^{\infty}$ such that $\tilde \phi_1 \varphi_1 \equiv \varphi_1$.
Then 
\begin{align}
  & | \eqref{5.53d3} |\notag \\
  =& \left| \sum_{J\ge 8} 2^{-2J}  \int \int  2^J \varphi_1(2^J y)  \chi(t^{-\frac 23} y)
\phi_2(t,x) 
(\partial_t \partial_1 u)(x-\theta y)  (T_j  w_J)(x-y)  \tilde \phi_1(2^Jy) \partial_{tt} P_J w e^p dy dx  \right| \notag \\
\le &
\sum_{J\ge 8} 2^{-2J} 
 \int \int  \Bigl( 2^{2J}|\varphi_1(2^J y)|^2
  \epsilon  \frac {|(T w_J)(x-y)|^2}  {\langle |x- y|-t\rangle^{1+\delta_1}} \notag \\
 & \qquad
 + C_{\epsilon}  {|\tilde \phi_1(2^Jy)|^2} {\langle |x-y|-t \rangle^{1+\delta_1} } |(\partial^2 u)(x-\theta y)|^2
 |\chi(t^{-\frac 23} y)|^2 |\phi_1(t,x)|^2 
 |(\partial_{tt} P_J w)(x)|^2 \Bigr) dx dy \notag \\
 \le & \sum_{J\ge 8} 2^{-2J}\mathrm{const} \cdot \epsilon\cdot \int {|Tw_J|^2(x)} q^{\prime} (|x|-t) dx  \notag \\
 & \qquad + \sum_{J\ge 8}2^{-4J} \cdot C_{\epsilon}^{(1)} \cdot \frac {E_5} t \| \partial_{tt} P_J w\|_2^2.
\end{align} 
In the above $\epsilon>0$ can be taken sufficiently small, and $C_{\epsilon}>0$, $C_{\epsilon}^{(1)}>0$
depend on $\epsilon$.  The term $\| \partial_{tt} P_J w\|_2^2$ can be 
 controlled with the help of Lemma \ref{Lem5.2S0}. Thus 
 \begin{align}
 |\eqref{5.53d3}| \le \mathrm{NICE}+ \frac {\mathrm{const} \cdot E_5} t \| \langle \nabla \rangle^{-1} \partial 
 \Gamma^{\le m+1} u \|_2^2.
 \end{align}
The term \eqref{5.53d4} is easier and can be estimated along similar lines. We omit the details.

Now observe
\begin{align}
 & \Bigl| \sum_{J\ge 8} 2^{-2J}  \int \int  2^J \varphi_1(2^J y) \chi(t^{-\frac 23} y)
\phi_2(t,x)
(\partial_t \partial_1 u)(x-\theta y)  (T_j  w_J)(x-y) \partial_t P_J w e^p dy dx\Bigr|
\notag \\
\lesssim &\; E_5^{\frac 12}  t^{-\frac 12}  \| \langle \nabla \rangle^{-1} \partial w \|_2^2.
\end{align}
Thus the contribution of the term
\eqref{e5.73S0} is acceptable for us.

This concludes the estimate of Subcase 3 and Case 1a. 

Case 1b: the piece 
 \begin{align}
 &\sum_{J\ge 8} 2^{-2J} g^{kij} \int P_J ( \partial_k u \partial_{ij} P_{\le J-4} w)
 \partial_t P_J w e^p dx \notag \\
= &\sum_{J\ge 8} 2^{-2J} g^{kij} \int P_J ( \partial_k \tilde P_J u \partial_{ij} P_{\le J-4}  w) 
 \partial_t P_J w e^p dx \\
 =
& \; \sum_{J\ge 8} 2^{-2J} g^{kij} \int \Bigl([P_J,  \partial_k \tilde P_J u ]\partial_{ij} P_{\le J-4} w\Bigr)
 \partial_t P_J w e^p dx. 
 \end{align}
This case can be similarly treated along the lines in Case 1a. To overcome the issue of summability
due to $P_{\le J-4}w$, one can make use of  Lemma \ref{Lem5.3Sa0} and Lemma \ref{Lem5.4S0}. 
For example, the analogue of \eqref{e5.66S1} is 
\begin{align} \label{5.64f}
\sum_{J\ge 8} 2^{-2J}  \int \int 
2^{2J} (\partial_j \varphi_1)(2^J y) \chi(t^{-\frac 23} y) \phi_1(t,x) \omega_k(x)
(\partial_t   \partial_1 u_J) (x-\theta y) 
( T_{i} w_{\le J-4})(x-y) \partial_t P_J w e^p  dy dx,
\end{align}
where $w_{\le J-4} =P_{\le J-4} w$ and $u_J=\tilde P_J u$.  In lieu of \eqref{e5.69S0}, we bound it as
\begin{align}
 & |\eqref{5.64f} | \notag\\
 \le &
 \sum_{J\ge 8} 2^{-2J} 
 \int \int  2^{2J}|(\nabla \varphi_1)(2^J y)| |\phi_1(t,x)| 
 \Bigl( \epsilon {|(T w_{\le J-4} )(x-y)|^2}  \langle |x-y|-t \rangle^{-1-\delta_1}\cdot 2^{-J\delta_2}
 \notag \\
 & \qquad
 + C_{\epsilon} \langle |x-y|-t \rangle^{1+\delta_1} 2^{J\delta_2} |(\partial^2 u_J)(x-\theta y)|^2
 |\chi(t^{-\frac 23} y)|^2
 |(\partial_t P_J w)(x)|^2 \Bigr) dx dy , 
\end{align} 
where $\delta_2>0$ is a small exponent.  The term containing $|Tw_{\le J-4} (x-y)|^2$ is clearly
manageable due to the decay factor $2^{-J\delta_2}$. 
For the second term,  by using Lemma \ref{Lem5.4S0}, we have
(for $|x| \sim t$, $|y|\ll |x|$)
\begin{align}
| \langle |x-\theta y| -t \rangle \partial^2 u_J (x-\theta y)  |^2 \lesssim t^{-1} 2^{-4J}.
\end{align}
Since $\langle |x-y| -t \rangle \lesssim \langle |x-\theta y| -t \rangle + |y|$, this term is under control.
Thus both terms are easily estimated. We omit further details.

 Case 1c: the piece 
 \begin{align}
 &\sum_{J\ge 0} 2^{-2J} g^{kij} \int P_J ( \partial_k u \partial_{ij} P_{\ge J+4} w)
 \partial_t P_J w e^p dx \notag \\
 =& \sum_{J\ge 0} 2^{-2J} g^{kij}
 \sum_{l\ge J+4} \int P_J ( \partial_k \tilde P_l u \partial_{ij} P_l w) 
 \partial_t P_J w e^p dx \notag \\
 =& \sum_{J\ge 0} 2^{-2J} g^{kij}
 \sum_{l\ge J+4} \int P_J ( (1-\phi_1) \partial_k \tilde P_l u \partial_{ij} P_l w) 
 \partial_t P_J w e^p dx   \label{5.76S0} \\
&+ \sum_{J\ge 0} 2^{-2J} g^{kij}
 \sum_{l\ge J+4} \int P_J ( \phi_1 \partial_k \tilde P_l u \partial_{ij} P_l w) 
 \partial_t P_J w e^p dx, \label{5.76S1}
 \end{align}
 where $\phi_1(t,x)=a(x/t)$ and $a \in C_c^{\infty}(\mathbb R^2)$ is a   radial  bump function 
 such that $a(x)=1$ for $ 0.9\le |x|\le 1.1$, and $a(x)=0$ for $|x|\le 0.8$ or $|x|\ge 1.2$.

 \begin{lem} \label{Lem5.3Sa0}
 We have for $l\ge 1$, 
 \begin{align}
 &\| (1-\phi_1) \partial P_l u \|_{\infty} \lesssim t^{-\frac 32} 2^{-3l }; \label{e5.78S0}\\
 &\| \phi_1 T P_l u \|_{\infty} \lesssim t^{-\frac 32} 2^{-3l}. \label{e5.79S0}
 \end{align}
 \end{lem}
 \begin{proof}
 Note that away from the light cone $\partial^2 \Gamma^{\le m-3} u$
 has $O(t^{-\frac 32})$ decay. The estimate \eqref{e5.78S0} then follows from a mismatch 
 estimate. For \eqref{e5.79S0}, we can take $T_1= \omega_1 \partial_t +\partial_1$ (
 the estimate for $T_2$ is similar) and observe that
 \begin{align}
 \| \phi_1  T_1 P_l u \|_{\infty}& \lesssim 
 \sum_{i,j=1}^2 \| \phi_1 T_1 \Delta^{-2}\partial_{ii} \partial_{jj} 
  P_l u \|_{\infty} \notag \\
 & \lesssim  2^{-3l} \sum_{i=1}^2 \| \phi_1 T_1 Q_l^{(i)} \tilde \partial^3 u \|_{\infty},
 \end{align}
 where $Q_l^{(i)}$ is modified frequency projection still localized to $|\xi| \sim 2^l$, 
 and $\tilde \partial = \partial_1 $ or $\partial_2$. 
 Note that
 \begin{align}
  \phi_1 T_1 Q_l^{(i)} \tilde \partial^3 u
 &= \phi_1 (\omega_1 \partial_t +\partial_1 ) Q_l^{(i)} \tilde \partial^3 u \notag \\
 &=[\phi_1 \omega_1, Q_l^{(i)} ] \partial_t \tilde \partial^3 u 
 + [\phi_1, Q_l^{(i)}] \tilde \partial^3 u + Q_l^{(i)} ( \phi_1 T_1 \tilde \partial^3 u).
 \end{align}
 Since
 \begin{align}
 \| \nabla (\phi_1 \omega_1) \|_{\infty} + \| \nabla \phi_1 \|_{\infty} \lesssim \frac 1 t,
 \end{align}
 the commutators $[\phi_1 \omega_1, Q_l^{(i)}]$, $[\phi_1, Q_l^{(i)}]$ are under control. The desired result  follows
 easily.
 \end{proof}
 
By using \eqref{e5.78S0}, it is not difficult to check that the contribution  of \eqref{5.76S0} is acceptable for us. For \eqref{5.76S1}, we note that
 \begin{align}
g^{kij} \partial_k \tilde P_l u \partial_{ij} P_l w
=g^{kij} (T_k 
\tilde P_l u \partial_{ij} P_lw - \omega_k \partial_t \tilde P_l u T_i \partial_j P_l w+ 
\omega_k \omega_i \partial_t \tilde P_l u T_j \partial_t P_l w).
\end{align}

 By \eqref{e5.79S0} and Lemma \ref{Lem5.2S0}, we have
 \begin{align}
  & \Bigl| 
  \sum_{J\ge 0} 2^{-2J} g^{kij}
 \sum_{l\ge J+4} \int P_J ( \phi_1 T_k \tilde P_l u \partial_{ij} P_l w) 
 \partial_t P_J w e^p dx \Bigr| \notag \\
 \lesssim &  \sum_{J\ge 0}2^{-2J} \sum_{l\ge J+4} 2^{-3l} t^{-\frac 32}
 \cdot \Bigl(
 2^{2l} \| \partial \langle \nabla \rangle^{-1} 
 \Gamma^{\le m+1} u\|_2 \Bigr) \cdot 2^J \| \partial \langle
 \nabla \rangle^{-1} w\|_2.
 \end{align}
 Clearly  the contribution of this term is acceptable for us.
 
 Next we estimate the piece
 \begin{align}
  & \Bigl| 
  \sum_{J\ge 0} 2^{-2J} 
  \sum_{1\le i,j\le 2, 0\le k\le 2} g^{kij}
 \sum_{l\ge J+4} \int P_J ( \phi_1 \omega_k \partial_t \tilde P_l u  T_i \partial_j  P_l w) 
 \partial_t P_J w e^p dx \Bigr| \notag \\
 \lesssim & \sum_{1\le i, j\le 2, 0\le k\le 2} \Bigl| 
  \sum_{J\ge 0} 2^{-2J} 
 \sum_{l\ge J+4} \int P_J ( \phi_1 \omega_k \partial_t \tilde P_l u  T_i \partial_j  P_l w) 
 \partial_t P_J w e^p dx \Bigr|.   \label{5.84S0}
  \end{align}
 In yet other words, we first treat the terms containing $T\nabla w$.
 
 \underline{Estimate of \eqref{5.84S0}}.  With no loss we take $k=1$, $i=1$, $j=1$. 
 Note that 
 \begin{align}
 & \Bigl| 
  \sum_{J\ge 0} 2^{-2J} 
 \sum_{l\ge J+4} \int P_J ( \phi_1 \omega_1 \partial_t \tilde P_l u  [T_1, \partial_1]  P_l w) 
 \partial_t P_J w e^p dx \Bigr|  \notag \\
 \lesssim &
 \sum_{J\ge 0}2^{-2J} \sum_{l\ge J+4}
 \frac 1t \| \partial \tilde P_l u \|_{\infty} \| \partial P_l w \|_2  \| \partial P_J w \|_2
 \lesssim t^{-\frac 32} E_5^{\frac 12} \| \partial \langle \nabla \rangle^{-1} w \|_2^2.
 \end{align}
 Thus the commutator piece is under control. 
 We now consider 
 \begin{align}
  &\Bigl| 
  \sum_{J\ge 0} 2^{-2J} 
 \sum_{l\ge J+4} \int P_J ( \phi_1 \omega_1 \partial_t \tilde P_l u  \partial_1 (T_1 P_l w) )
 \partial_t P_J w e^p dx \Bigr|  \notag \\
 \le & \; \mathrm{NICE}
 + \Bigl| 
  \sum_{J\ge 0} 2^{-2J} 
 \sum_{l\ge J+4} \int P_J ( \phi_1 \omega_1 \partial_1 \partial_t \tilde P_l u   T_1 P_l w)
 \partial_t P_J w e^p dx \Bigr|  \label{e5.86S0} \\
 & \qquad+\Bigl| 
  \sum_{J\ge 0} 2^{-2J} 
 \sum_{l\ge J+4} \int \partial_1P_J ( \phi_1 \omega_1  \partial_t \tilde P_l u   T_1 P_l w)
 \partial_t P_J w e^p dx \Bigr|. \label{e5.87S0}
 \end{align}
 For \eqref{e5.86S0}, we have
 \begin{align}
 & \eqref{e5.86S0} \notag \\
 \le &\;
  \sum_{J\ge 0} 2^{-2J}  \Bigl|
 \sum_{l\ge J+4} \int  \phi_1 \omega_1 \partial_1 \partial_t \tilde P_l u   T_1 P_l w
P_J( \partial_t P_J w e^p )dx \Bigr|  \notag \\
\lesssim & \sum_{J\ge 0} 2^{-2J} 
\sum_{l\ge J+4} \Bigl(\epsilon 2^{-2l}
\int |T P_l w|^2 q^{\prime} dx 
+ C_{\epsilon}2^{2l} \int \frac 1 {q^{\prime}}|\phi_1 \omega_1 \partial^2 \tilde P_l u|^2 
|P_J( \partial_t P_J w e^p )|^2 dx \Bigr), \label{e5.89S0}
\end{align}
 where $\epsilon>0$ can be taken sufficiently small, and $C_{\epsilon}>0$ depends on 
 $\epsilon$. 
 \begin{lem} \label{Lem5.4S0}
 We have
 \begin{align}
 \|  \phi_1 \langle r -t\rangle \partial^2 \tilde P_l u \|_{\infty} \lesssim t^{-\frac 12}
 2^{-2l}.
 \end{align}
 \end{lem}
 \begin{proof}
 By \eqref{2.9a1} and noting that $r\sim t$ (thanks to the cut-off $\phi_1$), we have
 \begin{align}
 \|\phi_1 \langle r -t\rangle \partial^2 \tilde P_l u \|_{\infty}
 &\lesssim \|  \partial\Gamma^{\le 1} \tilde P_l u \|_{\infty}
 + t \|\square \tilde P_l u \|_{\infty} \notag \\
 & \lesssim 2^{-2l} t^{-\frac 12}. 
 \end{align}
 \end{proof}
 By using Lemma \ref{Lem5.4S0}, it is not difficult to check that 
 \eqref{e5.89S0} is under control.  
 Thus \eqref{e5.86S0} is acceptable for us.  The estimate of \eqref{e5.87S0} is similar.
 We omit the details. This concludes the estimate of \eqref{5.84S0}.

 Next we estimate the piece
 \begin{align}
  & 
  \sum_{J\ge 0} 2^{-2J} 
  \sum_{1\le i\le 2, 0\le k\le 2} g^{ki0}
 \sum_{l\ge J+4} \int P_J ( \phi_1 \omega_k \partial_t \tilde P_l u  T_i \partial_t  P_l w) 
 \partial_t P_J w e^p dx. \label{5.85S0}
 \end{align}
The idea is to rewrite
\begin{align}
&\int P_J ( \phi_1 \omega_k \partial_t \tilde P_l u  T_i \partial_t  P_l w) 
 \partial_t P_J w e^p dx \notag \\
 = &\frac d {dt} 
 \Bigl( \int P_J ( \phi_1 \omega_k \partial_t \tilde P_l u  T_i   P_l w) 
 \partial_t P_J w e^p dx \Bigr) 
 - \int P_J ( \phi_1 \omega_k \partial_t \tilde P_l u  T_i  P_l w) 
 \partial_{tt} P_J w e^p dx \notag \\
 & \qquad -\int P_J ( \phi_1 \omega_k \partial_t \tilde P_l u  T_i   P_l w) 
 \partial_t P_J w e^p  \partial_t p dx 
  -\int P_J ( \partial_t ( \phi_1 \omega_k \partial_t \tilde P_l u)  T_i   P_l w) 
 \partial_t P_J w e^p dx.
\end{align} 
It is not difficult to check that all terms are under control.

Finally we note that the piece
\begin{align}
  & 
  \sum_{J\ge 0} 2^{-2J} 
   g^{kij}
 \sum_{l\ge J+4} \int P_J ( \phi_1 \omega_k \omega_i \partial_t \tilde P_l u  T_j \partial_t  P_l w) 
 \partial_t P_J w e^p dx
 \end{align}
can be estimated similarly. We omit the details.
This concludes the estimate of \eqref{5.76S1} and Case 1c.

Case 1d: the piece 
 \begin{align}
 &\sum_{0\le J\le 7} 2^{-2J} g^{kij} \int P_J ( \partial_k u \partial_{ij} P_{\le J+3} w)
 \partial_t P_J w e^p dx.
 \end{align}
Since  $0\le J\le 7$, it is not difficult to check that this case is under control.

Case 2: $|\alpha_1| \le \frac m2$, $|\alpha_2| \le m$ with $\alpha_1+\alpha_2\le \beta$, i.e.
the piece
\begin{align}
\sum_{J\ge 0} 2^{-2J}
g^{kij} \int P_J ( \partial_k \Gamma^{\alpha_1} u
\partial_{ij} \Gamma^{\alpha_2} u)
\partial_t P_J w e^p dx.
\end{align}
This case can again be treated by using the decomposition  (with no loss
consider the main case $J\ge 8$)
\begin{align}
&\sum_{J\ge 0} 2^{-2J}
g^{kij} \int P_J ( \partial_k \Gamma^{\alpha_1} u
\partial_{ij} \Gamma^{\alpha_2} u)
\partial_t P_J w e^p dx \notag \\
= & \sum_{J\ge 0} 2^{-2J}
g^{kij} \int P_J ( \partial_k \tilde P_J \Gamma^{\alpha_1} u
\partial_{ij} P_{\le J-3} \Gamma^{\alpha_2} u)
\partial_t P_J w e^p dx  \notag \\
&\quad+  \sum_{J\ge 0} 2^{-2J}
g^{kij} \int P_J ( \partial_k \Gamma^{\alpha_1} u
\partial_{ij} P_{[J-3, J+3]} \Gamma^{\alpha_2} u)
\partial_t P_J w e^p dx  \notag \\
&\quad + \sum_{J\ge 0} 2^{-2J}
g^{kij} \int P_J ( \partial_k \Gamma^{\alpha_1} u
\partial_{ij} P_{\ge J+4} \Gamma^{\alpha_2} u)
\partial_t P_J w e^p dx.
\end{align}
The estimates are similar to the quasilinear piece $\alpha_1=0$, $\alpha_2=\beta$.
We omit the details.

Case 3: $|\alpha_2| \le \frac m2$, $|\alpha_1| \le m$ with $\alpha_1+\alpha_2\le
\beta$, i.e.  the piece
\begin{align}
\sum_{J\ge 0} 2^{-2J}
g^{kij} \int P_J ( \partial_k \Gamma^{\alpha_1} u
\partial_{ij} \Gamma^{\alpha_2} u)
\partial_t P_J w e^p dx.
\end{align}
The situation is similar to the case $\alpha_1=\beta$, $\alpha_2=0$ which is discussed below. We omit
the details.

{Case 4: $\alpha_1=\beta$, $\alpha_2=0$}.  In this case we need to estimate
\begin{align}
\sum_{J\ge 0}
2^{-2J} g^{kij}
\int P_J( \partial_k w \partial_{ij} u) \partial_t P_J w e^p dx.
\end{align}

Case 4a: $J\ge 8$. We write 
\begin{align}
P_J(\partial_k w \partial_{ij} u) &= P_J ( \partial_k P_{\le J-3} w \partial_{ij} \tilde P_J u) +
P_J ( \partial_k P_{[J-2, J+2]} w \partial_{ij} u) + \sum_{l\ge J+3}
P_J( \partial_k P_l w \partial_{ij} \tilde P_l u),
\end{align}
where $\tilde P_l$ denotes the fattened Littlewood-Paley projection localized to $|\xi| \sim 2^l$. 

We shall sketch the details for the second term $P_J ( \partial_k P_{[J-2, J+2]} w \partial_{ij} u)$. The first and the third term can be treated along similar lines with the help of
Lemma \ref{Lem5.4S0}. Thus we only need to consider
\begin{align}
\sum_{J\ge 8}
2^{-2J} g^{kij}
\int P_J( \partial_k \tilde  w_J \partial_{ij} u) \partial_t P_J w e^p dx, 
\end{align}
where  $\tilde w_J= P_{[J-2,J+2]} w$.

Subcase 4a1: the regime $|r-t|\ge \frac 1{2} t$.   Choose a  radial  bump function 
$a \in C_c^{\infty}(\mathbb R^2)$ 
 such that $a(x)=1$ for $ 0.9\le |x|\le 1.1$, and $a(x)=0$ for $|x|\le 0.8$ or $|x|\ge 1.2$.
Define $\phi_1(t,x)= a(x/t)$.  We estimate the piece
\begin{align}  \label{5.12S0}
\sum_{J\ge 8} 2^{-2J} \int P_J ( \partial \tilde w_J \underbrace{(1-\phi_1) \partial^2 u}_{=:F_1}) \partial_t P_J  w e^p dx.
\end{align}
Observe that 
\begin{align}
\| F_1\|_{\infty} + \| \partial^2 F_1\|_{\infty} \lesssim t^{-\frac 32} E_5^{\frac 12}.
\end{align}
The contribution of this case is clearly acceptable.

Subcase 4a2: the regime $|r-t|<\frac 12 t$. We estimate the piece
\begin{align}
\sum_{J\ge 8}
2^{-2J} g^{kij}
\int P_J( \partial_k \tilde w_J  \phi_1 \partial_{ij} u) \partial_t P_J w e^p dx.
\end{align}

By using the null condition, we rewrite
\begin{align}
 g^{kij} \partial_k \tilde w_J  \partial_{ij} u
 = g^{kij} (T_k \tilde w_J T_i \partial_j u
 -\omega_i T_k \tilde w_J T_j \partial_t u
 -\omega_k \partial_t \tilde w_J T_i \partial_j u
 +\omega_k \omega_i \partial_t \tilde w_J T_j \partial_t u
 + \omega_i \omega_j T_k \tilde w_J \partial_{tt} u).
 \end{align}
 The first four terms all contain $T \partial u$.  To estimate them, it suffices for us to
 consider the general expression (below $h\in C^{\infty}$ corresponds
 to various expressions involving $\omega_k$, $\omega_i$ which are functions of
 the polar angle $\theta$)
\begin{align}
\sum_{J\ge 8} 2^{-2J} \int P_J ( \partial \tilde w_J  \underbrace{h(\theta) \phi_1 T\partial u}_{=:F_2} )
\partial_t P_J  w e^p dx.
\end{align}
Observe that 
\begin{align}
\|  F_2\|_{\infty} + \| \partial^2 F_2 \|_{\infty} \lesssim t^{-\frac 32}.
\end{align}
The contribution of this piece is clearly acceptable.

We then consider the main piece
\begin{align} \label{5.22S0}
\sum_{J\ge 8} 2^{-2J} \int P_J ( \underbrace{g^{kij} \omega_i \omega_j \phi_1}_{=:\phi_2}  
\partial_{tt} u
T_k \tilde w_J ) \partial_t P_J w e^p dx.
\end{align}
We estimate it as follows:
\begin{align}
|\eqref{5.22S0} |\le
\epsilon \sum_{J\ge 0} 2^{-2J} \int |T \tilde w_J|^2 q^{\prime} e^p dx
+ C_{\epsilon} \cdot  \sum_{J\ge 0} 2^{-2J} \int \frac 1 {q^{\prime} } 
|\partial_{tt } u|^2 |P_J(e^p \partial_t P_J w )|^2  dx,
\end{align}
where $\epsilon>0$ can be taken sufficiently small and $C_{\epsilon}>0$ depends 
on $\epsilon$.  Summing over $|\beta|=m+1$ and 
 taking $\epsilon>0$ sufficiently small, the first term above can be absorbed by
 the positive Alinhach term in \eqref{e5.8S0}. The second term can be bounded as
 \begin{align} \notag
 \mathrm{const} \cdot \frac 1 t  E_5 (u) \cdot \| \langle \nabla \rangle^{-1} \partial w \|_2^2
 \end{align}
 which is clearly acceptable for us.  
 
 Case 4b: $0\le J\le 7$. This is similar to the case $J\ge 8$ which some minor changes in numerology. We omit
 the details.

\bibliographystyle{abbrv}


\end{document}